\documentclass[journal]{IEEEtran}
\usepackage{amsmath,graphicx,color,cite,wrapfig,xcolor}

\newcommand{\beq}{\begin{equation}}
\newcommand{\eeq}{\end{equation}}

\newtheorem{lemma}{Lemma}
\newtheorem{theorem}{Theorem}
\def\remark{\addtocounter{remark}{1}\def\@currentlabel{\theremark}
\emph{Remark~\theremark}. } \makeatother
\newcounter{remark}

\newtheorem{corollary}{Corollary}[section]
\newtheorem{proposition}{Proposition}[section]

\def\smskip{\par\vskip 5 pt}

\def\QED{\hfill{\bf Q.E.D.}\smskip}

\usepackage{amsmath}
\usepackage{amssymb}
\usepackage{flushend}
\usepackage{cancel}
\usepackage[normalem]{ulem}
\usepackage{multirow}
\usepackage{caption}
\usepackage{makecell}
\usepackage{threeparttable}

\title{Hybrid Block Successive Approximation for One-Sided Non-Convex Min-Max Problems:  Algorithms and Applications}

\author{{Songtao Lu, Ioannis Tsaknakis,  Mingyi Hong and Yongxin Chen}\thanks{$^{*}$ SL and IT contributed equally, and the names are listed alphabetically. SL is with IBM Research AI, IBM Thomas J. Waston Research Center, Yorktown Heights, NY 10598, USA. IT and MH are with the Department of Electrical and Computer Engineering, University of Minnesota, MN 55455, USA. The authors are supported by NSF grants CMMI-172775, CIF-1910385 and by an ARO grant 73202-CS; YC is with the School of Aerospace Engineering, Georgia Institute of Technology, Atlanta, GA 30332, USA. He is supported by NSF under grant 1901599. This work is done when SL was a post-doctoral fellow at the University of Minnesota.

This paper has supplementary downloadable material available at http://ieeexplore.ieee.org., provided by the authors. The material includes an extension of the proposed algorithm and the respective convergence analysis. This material is 5 pages in size.}
		\thanks{Part of this work is based on a paper accepted by IEEE ICASSP 2019 \cite{luts19}}\vspace{-0.8cm}}
	
\begin{document}

\maketitle

%\vspace{-1cm}
\begin{abstract}
The min-max problem, also known as the saddle point problem, is a class of optimization problems which minimizes and maximizes two subsets of variables simultaneously. This class of problems can be used to formulate a wide range of signal processing and communication (SPCOM) problems. Despite its popularity, most existing theory for this class has been mainly developed for problems with certain special {\it convex-concave} structure. Therefore, it cannot be used to guide the algorithm design for many interesting problems in SPCOM, where various kinds of non-convexity  arise.

In this work, we consider a block-wise {\it one-sided} non-convex min-max problem, in which the minimization problem consists of multiple blocks and is non-convex, while the maximization problem is (strongly) concave. We propose a class of simple algorithms named Hybrid Block Successive Approximation (HiBSA),  which alternatingly performs gradient descent-type  steps for the minimization blocks and gradient ascent-type steps for the maximization problem. A key element in the proposed algorithm is the use of certain regularization and penalty sequences,  which stabilize the algorithm and ensure convergence. We show that HiBSA converges to some properly defined first-order stationary solutions with quantifiable global rates.  To validate the efficiency of the proposed algorithms, we conduct numerical tests on a number of problems, including the robust learning problem, the non-convex min-utility maximization problems, and certain wireless jamming problem arising in interfering channels.
\end{abstract}

%\vspace{-0.5cm}
\section{Introduction}
\label{sec:intro}
Consider the min-max (a.k.a. saddle point) problem below:
%\vspace{-0.5cm}
\begin{align}\label{eq:problem}
\begin{split}
\min_{{x}}\max_{{y}}&\quad f(x_1,x_2,\cdots, x_K, y) + \sum_{i=1}^{K}h_i(x_i) - g(y)\\
\mbox{s.t.} & \quad x_i\in \mathcal{X}_i, \; y\in \mathcal{Y}, \; i=1,\cdots, K
\end{split}
\end{align}
where $f:\mathbb{R}^{{N K}+M}\to \mathbb{R}$ is a continuously differentiable function; $h_i: \mathbb{R}^{N}\to \mathbb{R}$ and $g:\mathbb{R}^{M}\to \mathbb{R}$ are some convex possibly non-smooth functions; $x:=[x_1;\cdots; x_K]\in\mathbb{R}^{{N\cdot K}}$ and $y\in\mathbb{R}^M$ are the block optimization variables; $\mathcal{X}_i$'s and $\mathcal{Y}$ are some convex {and compact} feasible sets. We call the problem {\it one-sided} non-convex problem because we assume that $f(x,y)$ is non-convex w.r.t. $x$, and (strongly) concave w.r.t. $y$. For notational simplicity, we will use $\ell(x_1,x_2,\cdots, x_K, y)$ to denote the overall objective function for problem \eqref{eq:problem}.

Problem \eqref{eq:problem} is quite generic, and it arises in a wide range of signal processing and  communication (SPCOM) applications. We list of few of these applications below.

\vspace{-0.2cm}
\subsection{Motivating Examples in SPCOM}
\label{sec:format}
\vspace{-0.1cm}
\noindent {\bf Distributed non-convex optimization}: Consider a network of $K$ agents defined by a connected graph $\mathcal{G}\:=\{\mathcal{V},\mathcal{E}\}$  with $|\mathcal{V}|=K$, where each agent $i$ can communicate with its neighbors. A generic problem formulation that captures many distributed machine learning and signal processing problems can be formulated as follows \cite{mateos2010distributed,liao2015semi,Hajinezhad17inexact,Giannakis15}:
$$\min_{\{x_i\}} \sum_{i=1}^{K} f_{i}(x_i) + h_{i}(x_i),\; \|x_i-x_j\|\le c_{i,j}, \; (i,j)\; \mbox{are neighbors}$$
where each $f_{i} : \mathbb{R}^N \rightarrow \mathbb{R}$ is a non-convex, smooth function, $h_{i}: \mathbb{R}^N\rightarrow \mathbb{R}$ is a convex non-smooth regularizer, and $x_i\in\mathbb{R}^N$ is agent $i$'s local variable. Each agent $i$ has access only to $f_{i}$ and $h_{i}$. The non-negative constants $c_{i,j}$'s are predefined, and they can be selected to represent different levels of agreement among the agents \cite{Koppel_proximity}. 
Despite the fact that there have been a number of recent works on distributed non-convex optimization \cite{di2016next,tian18,Daneshmand18,Jiang2017,Lian2017,luh19}, the above problem formulation cannot be covered by any of these due to two main reasons: {\it i)} the nonsmooth regularizers $h_i$'s can be different across the nodes, invalidating the assumptions made in, e.g., \cite{di2016next} (which requires uniform regularizer across the nodes), and  {\it ii)} the partial consensus constraints are considered rather than the exact consensus where $c_{i,j}\equiv 0,\; \forall~i,j$. 

The above problem can be equivalently expressed  as:
\begin{align}\label{eq:distributed}
\begin{split}
\min_{x, \tilde{x}} & \quad  f({x})+h({x}) := \sum_{i=1}^{K} \left( f_{i}(x_{i}) + h_{i}(x_{i})\right)\; \\
\textrm{ s.t.}& \quad ({A} \otimes I_N){x} - \tilde{x} = 0, \; \tilde{x}\in \mathcal{Z} \subseteq \mathbb{R}^{{|\mathcal{E}|\cdot N}}
\end{split}
\end{align}
where ${x}:=[ x_{1}; \ldots; x_{K}]\in\mathbb{R}^{KN}$; ${A} \in \mathbb{R}^{|\mathcal{E}| \times K}$ is the incidence matrix, i.e., assuming that the edge $e$ is incident on vertices $i$ and $j$, with $i>j$ we have that $A_{ei}= 1, A_{ej}=-1$ and $A_{e\cdot}=0$ for all other vertices; $\otimes$ denotes the Kronecker product. $\tilde{x}\in {\mathbb{R}^{|\mathcal{E}|  N}}$ is the auxiliary variable representing the difference between two neighboring local variables;
the feasible set {$\mathcal{Z}$} represents the bounds on the size of the {differences.}
Using duality theory we can introduce the Lagrangian multiplier vector $y$ and rewrite the above problem as:
\begin{equation}\label{eq:reform}
\min_{{x}\in\mathbb{R}^{KN}, \tilde{x}\in \mathcal{Z}}\max_{{y}\in\mathbb{R}^{{|\mathcal{E}|\cdot N}}} f({x})+h({x})+\left\langle {y},{(A \otimes I_N)}{x} -\tilde{x}\right\rangle.
\end{equation}
See Sec. \ref{sub:opt} for detailed discussion on this reformulation and its relationship with \eqref{eq:distributed}. Clearly \eqref{eq:reform} is in the form of \eqref{eq:problem}.

\noindent {\bf Robust learning over multiple domains}: In \cite{qian2018robust} the authors introduce a robust learning framework, in which the training sets from $M$ different domains are used to train a machine learning model. Let $\mathcal{S}_{m}\:=\lbrace({s}_{i}^{m},t_{i}^{m}) \rbrace,\,1\leq m\leq M$ be the individual training sets with ${s}_{i}^{m} \in \mathbb{R}^{N}$, {${t}_{i}^{m} \in \mathbb{R}$}; ${x}$ be the parameter of the model we intent to learn, ${\ell}(\cdot)$ a non-negative loss function, and {$f_{m}({x}) = \frac{1}{| \mathcal{S}_{m}|} \sum_{i=1}^{|\mathcal{S}_{m}|}  {\ell}({s}_{i}^{m},t_{i}^{m},{x})$} is the (possibly) non-convex empirical risk in the $m$-th domain. The following problem formulates the task of finding the parameter ${x}$ that minimizes the empirical risk, while taking into account the worst possible distribution over the $M$ different domains:
\begin{equation} \label{eq:robust}
\min\limits_{x} \max\limits_{y\in\Delta} \; y^{T}F(x)-\dfrac{\lambda}{2}D(y||q)
\end{equation}
where $F({x}):=[f_{1}({x});\ldots;f_{M}({x})]\in\mathbb{R}^{M\times 1}$; ${y}$ describes the adversarial distribution over the different domains; $\Delta:=\{{y}\in\mathbb{R}^{M} \mid 0\leq y_{i}\leq 1,\,i=1,\ldots,M,\,\sum_{i=1}^{M} y_{i}=1\}$ is the standard simplex; $D(\cdot)$  is some distance between probability distributions,  $q$ is some prior probability distribution, and $\lambda>0$ is some  constant. The last term in the objective function represents some regularizer that imposes structures on the adversarial distribution.

\noindent {\bf Power control and transceiver design problem}: Consider a problem in wireless transceiver design, where  $K$ transmitter-receiver pairs  transmit over $N$ channels to maximize their {\it minimum} rates. User $k$ transmits messages with power $x_k:=[x^1_k; \cdots; x^N_k]$, and its rate is given by (assuming  Gaussian signaling):
$$R_k(x_1, \ldots, x_K) = \sum_{n=1}^{N}\log\Big(1 + \dfrac{a^n_{kk} x^n_k}{\sigma^2 + \sum_{\ell=1, \ell\neq k}^{K} a^n_{\ell k} x^n_\ell}\Big),$$
which is a non-convex function on $x:=[x_1;\cdots; x_K]$. Here  $a^n_{\ell k}$'s denote the channel gain between the pair $(\ell,k)$ on the $n$th channel, and $\sigma^2$  is the noise power.  Let $\bar{x}$ denote the power budget for each user, then the classical max-min fair power control problem is: $\max\limits_{x \in \mathcal{X}} \min\limits_{k}  R_{k}(x) $, where  $\mathcal{X}:=\lbrace x\mid 0 \leq \sum_n x^n_k \leq \bar{x} , \forall~k\rbrace$ denotes the feasible power allocations. The above max-min rate problem can be equivalently formulated as
\eqref{eq:problem} (see Sec. \ref{sub:opt} for details) \footnote{A minus sign is added to equivalently transform to the min-max problem.}:
\begin{equation}\label{eq:max-min-fair}
\min\limits_{x\in\mathcal{X}} \max\limits_{y \in \Delta}\quad  \sum\limits_{k=1}^{K} -R_{k}(x_1,\cdots, x_K) \times y_k,
\end{equation}
where the set $\Delta\subseteq \mathbb{R}^K$ is again the standard simplex.

{A closely related problem is the coordinated beamforming design in a (multiple input single output) MISO interference channel. In this case the target is to find the optimal beamforming vector for each user in order to maximize some system utility function under the total power and outage probability constraints \cite{cobf}. When the min-rate utility is used, this problem can be formulated  as
\begin{align}
\max\limits_{x_i \in \mathbb{C}^{N_t}, \forall i} \min\limits_{i}  R_{i}(\{x_{k}\}) \quad  \mbox{\rm s.t.} \; \|x_i\|^{2} \leq \bar{p}, \forall i
\end{align}
where $x_i$ is the transmit beamformer, $N_t$ is the number of antennas. Also,  $R_{i}(\{x_{k}\}) = \log_2(1+\xi_i(\{x_{k}\}_{k \neq i}) x_{i}^{H}Q_{ii}x_{i})$, where $\xi_i$ incorporates the outage constraints {and the cross-link interference, while} $Q_{ii}$ denotes the covariance matrix of the channel between the $i$th transmitter-receiver pair.
}

	 For other setups, similar min-max problems can be formulated, some of which can be solved optimally (e.g., power control \cite{Foschini93,zander92a, zander92b}, transmitter density allocation \cite{luswa19}, or certain MISO beamforing \cite{bengtsson01,Wiesel06}). But for general multi-channel and/or MIMO interference channel, the corresponding problem is NP-hard \cite{liu13maxminTSP}. Many heuristic algorithms are available for these problems \cite{liu13maxminTSP,Razaviyayn12maxmin,sun14downlink,luwa18}, but they are all designed for special problems, and often require repeatedly invoking computationally expensive general purpose solvers.
	  For computational tractability, a common approach is to perform the following approximation of the min-rate utility \cite{Li15outtage}:
	  %{\setlength{\abovedisplayskip}{-0.3pt}\setlength{\belowdisplayskip}{-0.3pt}
	  	\begin{align}\label{eq:min:max:approx}
	  	\min_{i} r_i \approx -1/\gamma \log_2 \sum_{i=1}^{N} 2^{-\gamma r_i}.
	  	\end{align}
	 However such an approximation procedure can introduce significant rate losses, as will be seen in Sec. \ref{sec:majhead}.

\noindent {\bf Power control in the presence of a jammer:} Consider an extension of the power control problem, where a jammer participates in a $K$-user $N$-channel interference channel transmission \cite{gohary2009generalized}. Differently from a regular user, the jammer's objective is to reduce the sum-rate of other users by properly transmitting noises.
Let $y^n$ denote the jammer's transmission on the $n$th channel, then one can formulate the following sum-rate maximization-minimization problem:
\begin{align}\label{eq:jammer}
\min\limits_{x \in \mathcal{X}} \max\limits_{{y} \in \mathcal{Y}}\hspace{-0.1cm}
\sum\limits_{(k,n)}\hspace{-0.1cm}  -\log \left( 1 + \dfrac{a^n_{kk} x_{k}^{n}}{\sigma^{2} + \sum_{j=1,j \neq k}^{{K}} a_{jk}^{n} x_{j}^{n} + a_{0k}^{n}y^{n}}\right),
\end{align}
where $x_{k}$ and $y$ are  the power allocation of user $k$ and the jammer, respectively; the set $\mathcal{X}:=\mathcal{X}_{1}\times\cdots\times\mathcal{X}_{K}$, where $\mathcal{X}_{k}$ are defined similarly as before.

\subsection{Related Work}
Motivated by these applications, it is of interest to develop efficient algorithms for solving these problems with theoretical convergence guarantees. In the optimization community, there has been a long history of studying min-max optimization problems.
When the problem is convex in $x$ and concave in $y$, algorithms have been developed which can solve the convex-concave saddle problem optimally; see \cite{Nedi09,hizh18,chen2014optimal,daskalakis2018limit,daskalakis2017training} and the references therein. However, when the problem is non-convex, the convergence behavior of such alternating type algorithms has not been well understood.

Although there are many recent works on the non-convex minimization problems \cite{bertsekas99}, only a few of them have been focused on the non-convex min-max problems.
An optimistic mirror descent algorithm is proposed in \cite{mertikopoulos2018omd},  and its convergence to a {saddle} point is established under certain strong coherence assumptions.
In \cite{qian2018robust}, algorithms for robust optimization are proposed, where the $x$ problem is unconstrained, and $y$ linearly couples with a non-convex function of $x$ [cf. \eqref{eq:robust}].
In \cite{qihang18}, a proximally guided stochastic mirror descent method (PG-SMD) is proposed, which provably converges to an approximate stationary point of the outer minimization problem.
An oracle based non-convex stochastic gradient descent for generative adversarial networks (GAN) is proposed in \cite{same18}, where the algorithms solve the maximization subproblem up to some small error. Moreover, in\cite{sanjabi2018solving, nouiehed2019solving} a multi-step GDA scheme is introduced, where the maximization problem is approximately solved using a number of gradient ascent steps. In \cite{gan_minmax} the convergence of a primal-dual algorithm to a first-order stationary point is established for a class of GAN problems  formulated as a special min-max optimization problem where the coupling term is linear w.r.t the discriminator.
More recently, in \cite{lin2019gradient} it has been shown that GDA can converge to a stationary point of the outer minimization problem in the (strongly) concave case, under certain conditions. Under the same optimality criterion and assuming that the inner problem is concave, \cite{thekumparampil2019efficient}  proves convergence using a proximal dual implicit accelerated gradient method.

It is worth noting that, in the works discussed above, different optimality criteria are often utilized. Since these conditions are not equivalent to each other, one cannot directly compare the convergence guarantees of algorithms that reach these criteria. On the other hand, these optimality criteria often share some interesting implicit connections. For example, it can be shown that, no matter if the inner maximization problem is strongly concave or concave, as long as a point $(x^{\ast},y^{\ast})$ is an (exact or approximate) stationary point defined in this current work [see \eqref{eq:stationarity}], then it is also an (exact or approximate, respectively) stationary point in the sense defined in \cite{qihang18}; see \cite{lin2019gradient} for detailed discussions.
In Table \ref{tb:rate} we provide a summary of some algorithms discussed above, including the complexity and the respective optimality criterion.

\begin{table*}
{\footnotesize
\begin{center}
\begin{threeparttable}
\begin{tabular}{|c|c|c|c|c|c|c|}
\hline
\textbf{Algorithm} & \textbf{ Optimality Criterion} & \textbf{Det./ St.} & \multicolumn{2}{c|}{\textbf{Assumptions}} &  \textbf{ Gradient Complexity} \tnote{1} \\
\hline
\multirow{2}{*}{ \makecell{ \textit{Multi-Step GDA}  \cite{sanjabi2018solving,nouiehed2019solving}} }  &  \multirow{2}{*}{ \makecell { \\  \tnote{2}~~ 1st order Nash equilibrium}} &  \multirow{2}{*}{ \makecell { \\  Det.}}  & \multicolumn{2}{c|}{\makecell{ $f$ NC in $x$/Polyak-Lojasiewicz in $y$ \\ $h(x)=0, \; g(y)=0$}} &   $\mathcal{O}(\epsilon^{-2} \log(\frac{1}{\epsilon}))$ \\ \cline{4-6}
& & &   \multicolumn{2}{c|}{\makecell{ $f$ NC in $x$/concave in $y$  \\ $h(x)=0, \; g(y)=0$}} &   $\mathcal{O}(\epsilon^{-3.5} \log(\frac{1}{\epsilon})))$ \\ \hline
\multirow{2}{*}{ \makecell{ \\ \textit{Robust optimization} \cite{qian2018robust} } }  &  \multirow{2}{*}{ \makecell { \tnote{3}~~1st order SP for min. problem \\ Optimality gap for max. problem}} &  \multirow{2}{*}{ \makecell { \\   St.}}  & \multicolumn{2}{c|}{\makecell{ $f$ NC in $x$/linear in $y$  \\ $h(x) =0$, \; $g$ convex, {smooth}}} &   $\mathcal{O}(\epsilon^{-6})$\\ \cline{4-6}
& & &  \multicolumn{2}{c|}{\makecell{ $f$ NC in $x$/linear in $y$  \\ $h(x) =0$, \; $g$ str. convex, {smooth}}} &   $\widetilde{\mathcal{O}}(\epsilon^{-2} + \epsilon^{-4})$ \\ \hline

\multirow{4}{*}{\makecell{ \\ \\ \\ \textit{PG-SMD/ PGSVRG}\cite{qihang18}} }   & \multirow{4}{*}{ \makecell { \\ \\ \\  \tnote{4}~~1st order SP for min. problem}} & \multirow{4}{*}{ \makecell { \\  \\ \\  St.}} &  
\multirow{2}{*}{\makecell{ {$f = \frac{1}{n}\sum_{i=1}^{n}y^{T}c_i(x)$, $c_i(x)$ NC} \\ \\ $f$ NC in $x$, $f$ concave in $y$ }} & \makecell{\\  {$g$ str. convex} }  & \makecell{  $\widetilde{\mathcal{O}}(n\epsilon^{-2} + \epsilon^{-4})$}  \; \\ \cline{4-6}
& & & & \makecell{\\ {$g$ convex}} & \makecell{$\widetilde{\mathcal{O}}(\epsilon^{-6} )$} \\ \cline{4-6}
& & & \multirow{2}{*}{\makecell{$f = \frac{1}{n} \sum_{i=1}^{n} f_{i}$ \\[1ex]  $f$  NC in $x$ \\[0.11ex]{$f$ concave in $y$} }} & \makecell{\\  {$g$ str. convex}}  &  \makecell{$\widetilde{\mathcal{O}}(n\epsilon^{-2})$ } \\ \cline{5-6}
& & & & {\makecell{\\ {$g$ convex}}} &\makecell{ $\widetilde{\mathcal{O}}(n\epsilon^{-2} + \epsilon^{-6} )$}
\\ \hline

\multirow{3}{*}{ \makecell { \\ \\ \textit{HiBSA} (our work)}}  & \multirow{2}{*}{ \makecell { \\ \\   \tnote{5}~~ 1st order SP (def. \eqref{eq:stationarity}) }} & \multirow{3}{*}{ \makecell { \\ Det.}} & \multicolumn{2}{c|}{\makecell{ \\$f$ NC in $x$/strongly concave in $y$ }} &  $\mathcal{O}(\epsilon^{-2})$ \\ \cline{4-6}
& & & \multicolumn{2}{c|}{\makecell{ \\ $f$ NC in $x$/ linear in $y$}}  &  $\widetilde{\mathcal{O}}(\epsilon^{-4})$ \\ \cline{4-6}
& & & \multicolumn{2}{c|}{\makecell{ \\  \tnote{6}~~$f$ NC in $x$/concave in  $y$}}  &    $\widetilde{\mathcal{O}}(\epsilon^{-4} \log(\frac{1}{\epsilon}))$  \\ \hline
\end{tabular}
\vspace{0.2cm}
{\small} $^1$ Gradient complexity refers to the total  number of gradient evaluations required to reach an $\epsilon$-stationary solution, where the definition of stationarity can vary for different works.; 
	 \; { $^2${ The approximate stationarity condition for $(\tilde{x},\tilde{y})$ is $- \min_{x} \langle \nabla_{x} f(\tilde{x},\tilde{y}),x-\tilde{x} \rangle \leq \epsilon, \forall x \in \mathcal{X} : \|x - \tilde{x} \| \leq 1$ and $\max_{y}\langle \nabla_{y} f(\tilde{x},\tilde{y}),y-\tilde{y} \rangle \leq \epsilon, \forall y \in \mathcal{Y} : \|y - \tilde{y} \| \leq 1$} ;  \; $^3$For the minimization problem, convergence is established using the stationarity gap, i.e {$\| \nabla_{x} f(x,y)\| \leq \epsilon$}, while for the maximization problem the optimality condition { $\max_{y^{\prime} \in \mathcal{Y}}f(x,y^{\prime}) - f(x,y) \leq \epsilon$} is used;  \;  $^4$The optimality measure is the norm of the gradient of the Moreau envelope of the minimization problem, i.e { $\| \nabla \phi_{\gamma}(x) \| \leq \epsilon$} with $\phi_{\gamma}(x) = \min_{z} \{ \phi(z) + (1/2\gamma) \| z-x \|^{2} \}$ and $\phi(z) = \max_{y} \{ f(z,y) - g(y)\}$.}; \; { $^5$ A point $(\tilde{x},\tilde{y})$ such that $\| \mathcal{G}^{\beta}_{\rho}(\tilde{x},\tilde{y}) \| \leq \epsilon$; \; $^6$  This complexity is obtained for the algorithm given in the  supplementary document of the current work  
\cite{ieee_minmax_supp}. }
\end{threeparttable}
\end{center}
%\vspace{-0.2cm}
\caption{\small{Summary of algorithms for $\min\limits_{x}\max\limits_{y}\; f(x, y) + h(x) - g(y)$, where { $f$ is a smooth function, while} $h$ and $g$ are considered convex { non-smooth} functions unless otherwise stated. Note that in the 3rd column we characterize the type of the { algorithm}, i.e deterministic (Det.) or stochastic (St.). Moreover, we use the abbreviations NC for non-convex, SP for stationary point, str. for strongly, { min. for minimization and max. for maximization.} }}
\label{tb:rate}
\vspace{-0.4cm}}
\end{table*}

\vspace{-0.2cm}
\subsection{Contribution of this work}
In this work, we design effective algorithms for the min-max problem by adopting the popular block alternating minimization/maximization strategy.  The studied problems allow non-convexity and non-smoothness in the objective, as well as non-linear coupling between variables. The algorithm proposed in this work is named the {\it Hybrid Block Successive Approximation} (HiBSA) algorithm, because it updates the variables block by block,  where each block is optimized using a strategy similar to the idea of successive {\it convex} approximation (SCA) \cite{meisam15BSUMsurvey} -- except that to update the $y$ block, a {\it concave} approximation is used (hence the name ``hybrid''). Despite the fact that such a block-wise alternating optimization strategy is simple and easy to implement [for example it has been used in the popular block successive upper bound minimization (BSUM) framework \cite{razaviyayn14thesis,meisam15BSUMsurvey} for {\it minimization-only} problem], it turns out that having the maximization subproblem invalidates all the previous analysis for {\it minimization-only}  algorithms. 

The main contributions of this paper are listed as follows. First, a number of applications in SPCOM have been formulated in the framework of non-convex, one-sided min-max problem \eqref{eq:problem}. Second, based on different assumptions on how $x$ and $y$ variables are coupled, as well as whether the $y$ problem is strongly {concave or merely concave},  three different types of min-max problems are studied. For each of the problem class, a simple single loop  algorithm is presented, together with its convergence guarantees\footnote{{{In addition to the algorithm presented in the main text, we also provide an alternative double-loop algorithm for the case where the $y$ problem is concave, in the {supplementary document of this article
\cite{ieee_minmax_supp}}.}}}. The major  benefits of using the block successive approximation strategy are twofold: 1) each subproblem can be solved effectively, and 2) it is relatively easy to integrate many existing algorithms that are designed for only  solving minimization problems (such as those based on the BSUM framework \cite{razaviyayn14thesis,meisam15BSUMsurvey}). Finally, extensive numerical experiments are conducted for selected applications from SPCOM to validate the proposed algorithms.

Overall,  to the best of our knowledge this is the first time that the convergence  of the alternating block successive approximation type algorithm is rigorously analyzed for the (one-sided) non-convex min-max problem \eqref{eq:problem}.

{\textbf{Notation} The notation $\|\cdot\|$ denotes the vector $2$-norm $\| \cdot \|_{2}$; $\otimes$ denotes the Kronecker product; $I_{N}$ is the $N \times N $ identity matrix; $\langle \cdot, \cdot \rangle$ is the Euclidean inner product; $\mathcal{I}_{\mathcal{X}}(x)$ denotes the indicator function on set $\mathcal{X}$; in case the subscript is missing the set is implied by the context; $[K]:=\{1,\cdots, K\}$. Finaly, the notation $\tilde{\mathcal{O}}$ denotes big O notation $\mathcal{O}$ up to some logarithmic factor.
}

\vspace{-0.2cm}
\section{The Proposed Algorithms and Analysis}
\label{sec:algorithm}
\vspace{-0.1cm}
In this section, we present our main algorithm.
Towards this end, we will first make a number of blanket assumptions on problem \eqref{eq:problem}, and then present the HiBSA algorithm in its generic form. We will then discuss in detail about various algorithmic choices, as well as major challenges in the analysis.

Let the superscript $r$ denote iteration number. For notational simplicity, we will define the following:
\begin{subequations}
\begin{align}\label{eq:def:w}
w^{r+1}_i &:= [x^{r+1}_1; \cdots, x^{r+1}_{i-1}, x^r_i, \cdots x^r_K]\in\mathbb{R}^{NK},\\
w^{r+1}_{-i} &:= [x^{r+1}_1; \cdots, x^{r+1}_{i-1}, x^r_{i+1}, \cdots x^r_K]\in\mathbb{R}^{N(K-1)},\\
x_{-i} &:= [x_1; \cdots, x_{i-1}, x_{i+1}, \cdots x_K]\in\mathbb{R}^{N(K-1)}.
\end{align}
\end{subequations}
Throughout the paper, we will assume that problem \eqref{eq:problem} satisfies the following blanket assumption.

\noindent{\underline{\sf{Assumption A}}.} The following conditions hold for \eqref{eq:problem}:
\begin{itemize}
\item[A.1] $f:\mathbb{R}^{KN+M}\to \mathbb{R}$ is continuously differentiable; The feasible sets $\mathcal{X}  = \mathcal{X}_{1} \times \ldots \times \mathcal{X}_{K}$ and $\mathcal{Y} \subseteq \mathbb{R}^{M}$ are convex and compact. Further $\ell(x,y)$ is lower bounded, that is, $\ell(x,y)\ge \underline{\ell},\; \forall~x\in \mathcal{X}, y\in \mathcal{Y}$.

\item[A.2] $h_i(\cdot)$'s and $g(\cdot)$  are convex and non-smooth functions;

\item[A.3]$f$ has Lipschitz continuous gradient with respect to (w.r.t.) $x_{i}$ for every $i$ with constant $L_{x_{i}}$, that is:
\vspace{-0.10cm}
\begin{align}
\hspace{-0.9cm}\| \nabla_{x_{i}} f(\bar{z}) - \nabla_{x_{i}}f(z) \| \leq L_{x_{i}} \| \bar{z} - z \|, \forall \bar{z}, z \in  \mathcal{X} \times \mathcal{Y};
\end{align}
Furthermore, $f$ has Lipschitz continuous gradient w.r.t. $y$ with constant $L_{y}$, that is:
\vspace{-0.10cm}
\begin{subequations}
\begin{align}
\hspace{-0.9cm}{\| \nabla_{y} f(\bar{z}) - \nabla_{y}f(z) \| \leq L_{y} \| \bar{z} - z \|, \forall \bar{z}, z \in  \mathcal{X} \times \mathcal{Y}} \label{eq:Ly}.
\end{align}	
\end{subequations}
\vspace{-0.6cm}
\end{itemize}

{Next we describe the proposed HiBSA algorithm.
		\begin{center}
			\fbox{\small
				\vspace{-1cm}
				\begin{minipage}{0.95\linewidth}
					\textsf{\bf Hybrid Block Successive Approximation (HiBSA) Algorithm}\\
					At each iteration $r=1, 2, 3,\cdots$ \\
					\; {\bf [S1].} For $i=1,\cdots, K$, perform the following update:
					\begin{align}\label{eq:x-update}
			\hspace{-0.3cm}	x_i^{r+1}\hspace{-0.1cm}	 =&\arg\hspace{-0.1cm}	\min_{x_i\in\mathcal{X}_i}\hspace{-0.1cm}	 U_{i}(x_{i};  w^{r+1}_i, y^r) + h_i(x_i) \hspace{-0.05cm}	 +\hspace{-0.05cm}	  \frac{\beta^r}{2}\|x_i-x_i^r\|^2
					\end{align}		
					
					\; {\bf [S2].} Perform the following update for the $y$-block:
					\begin{align}\label{eq:y-update}
					y^{r+1}=&\arg\max_{y\in\mathcal{Y}}\;  {U_{y}(y; x^{r+1}, y^r)}  - g(y) -\frac{\gamma^r}{2}\|y\|^2;
					\end{align}
						
					\; {\bf [S3].} If  converges, stop; otherwise, set  $r=r+1$, go to  {\bf [S1]}.
				\end{minipage}
			}
		\end{center}
		
Note that $\{\beta^r\ge 0\}$ and $\{\gamma^r\ge 0\}$  are some algorithm parameters, whose values will be specified shortly in the next section. Properly designing the regularization sequence $\{\gamma^r\}$ is the key to ensure that the algorithm works when the $y$ problem is concave but not {\it strongly} concave.
Further, each $U_i(\cdot; w,y):\mathbb{R}^{N}\to \mathbb{R}$ [resp. $U_y(\cdot, w, y)$]
is some approximation function of $f(\cdot, x_{-i},y)$ [resp. $f(x,\cdot)$].  For the  $U_i(\cdot)$'s the following assumptions hold.
}

\noindent{\underline{\sf{Assumption B}}.} Each $U_i(\cdot)$ satisfies the following conditions:
\begin{itemize}
	\item[B.1] {\bf (Strong convexity).}  Each $U_{i}(\cdot; w, y)$ is strongly convex with modulus $\mu_{i}>0$: {\small
	\begin{align*}
	& U_{i}(x_i; w, y) - U_{i}(z_i; w, y) \ge \langle \nabla_{z_i} U_{i}(z_i; w, y), x_i-z_i \rangle \nonumber\\
	&\quad + \frac{\mu_i}{2}\|x_i-z_i\|^2, \; \; \forall~w \in \mathcal{X}, y\in \mathcal{Y}, x_i, z_i\in \mathcal{X}_{i}.
	\end{align*}}
	\item[B.2] {\bf (Gradient consistency).} Each $U_{i}(\cdot; w,y)$ satisfies:
	\begin{align*}
	\hspace{-0.8cm} \nabla_{z_i} U_{i}({z_i}; x, y)\Bigr|_{z_i=x_i} = \nabla_{x_i} f(x,y), \; \forall~i, \; \forall~x\in \mathcal{X}, \; y\in \mathcal{Y}.
	\end{align*}
	\item[B.3] {\bf (Tight upper bound).} Each $U_{i}(\cdot; w,y)$ satisfies:
	\begin{align*}
	U_{i}(z_i; x, y) & \ge f(x,y), \;  \mbox{and}\; U_{i}(x_i; x, y) = f(x,y), \nonumber\\
	&\quad \; \forall~ x\in \mathcal{X}, y\in \mathcal{Y}, \; z_i\in \mathcal{X}_{i}.
	\end{align*}
		\item[B.4] {\bf (Lipschitz gradient).} Each $U_{i}(\cdot; w,y)$ satisfies:
		\begin{align*}
		&\|\nabla U_{i}(z_i; w, y) - \nabla U_{i}(v_i; w, y)\|\le L_{u_i}\|v_i-z_i\|, \\
		&\quad \; \forall~ w\in \mathcal{X}, y\in \mathcal{Y}, v_i, z_i\in \mathcal{X}_{i}.
		\end{align*}
\end{itemize}

%}

Clearly, the $x$ update step {\bf[S1]} closely resembles the  BSUM  algorithm \cite{razaviyayn2013unified,meisam15BSUMsurvey}, which is designed for minimization problems. Similarly as in BSUM, approximation functions are used to
simplify the update for each subproblem; see \cite{meisam15BSUMsurvey} for a number of such
functions often used in signal processing
applications.

{However, a key difference from the BSUM, or for that matter, all successive convex
approximation (SCA) based algorithms such as the
inexact flexible parallel algorithm {(FLEXA)}
\cite{scutari13decomposition,scutari13flexible,Facchinei15}, the concave-convex procedure (CCCP) \cite{Yuille:2003:CP:773700.773708}, is the presence of
the {\it ascent} step in {\bf [S2]}. [Songtao: Too long consider to revise]}
This step is needed to deal with the inner maximization problem,
but unfortunately the use of it invalidates the existing analyses for SCA-type algorithms, because all of them critically
depend on consistently achieving some form of descent as the algorithms progress. As a result, how to properly
implement and analyze the proposed algorithm represents a major challenge.

 We note that it is not straightforward to design algorithms for one-sided non-convex min-max problem, as compared with non-convex minimization problems. For the former problem, simple algorithms like gradient descent-ascent can diverge (see Example 1 or \cite{daskalakis2018limit}), but if we specialize such an algorithm to the latter problem (which becomes the well-known gradient descent), then it will converge to a second-order stationary solution {\cite{pmlr-v49-lee16}.
 We refer the readers to a few recent works {\cite{lin2019gradient} \cite{thekumparampil2019efficient}} for more discussions.

\noindent	\underline{\textsf{Example 1.}}\cite{daskalakis2018limit} Consider a special case of problem \eqref{eq:problem}, where $K=1$ (a single block variable), and $A$ is a randomly generated matrix of size $N\times M$: $\min_{x\in\mathbb{R}^N}\max_{y\in\mathbb{R}^M}\; y^T A x.$
Let us apply a special case of HiBSA algorithm by utilizing the following approximation function:
{	\begin{align*}
	\hspace{-0.3cm}U_1(v; w,y) & =  y^T \hspace{-0.1cm} A v + \frac{\eta}{2}\|v-{w}\|^2\nonumber\\
	 U_y(u; x, y) & =  u^T A x - \frac{1}{2\lambda}\|u-y\|^2.
	\end{align*}}
Letting $\gamma^r=0$ and $\beta^r=0$ for all $r$, the HiBSA becomes an alternating gradient descent-ascent algorithm
	\begin{align}
	x^{r} = x^{r-1}- \frac{1}{\eta}A^T y^{r-1}, \quad y^{r} = y^{r-1} + {2\lambda} A x^{r}, \; \forall~r\label{eq:alter:min:max}.
	\end{align}

Unfortunately, one can verify that for almost any $A$, {\it regardless} the choices of $\eta, \lambda$,  \eqref{eq:alter:min:max} will not converge to the desired solution satisfying: $A^T y^*=0$ and $A x^*=0$; see Fig. \ref{fig:diverge}. This is because the linear system describing the dynamics of the vector $(x^r, y^r)$ is always unstable.
\hfill$\blacksquare$
\begin{wrapfigure}{r}{0.23\textwidth}
	\vspace{-0.7cm}
	\begin{center}
		\includegraphics[width=1\linewidth]{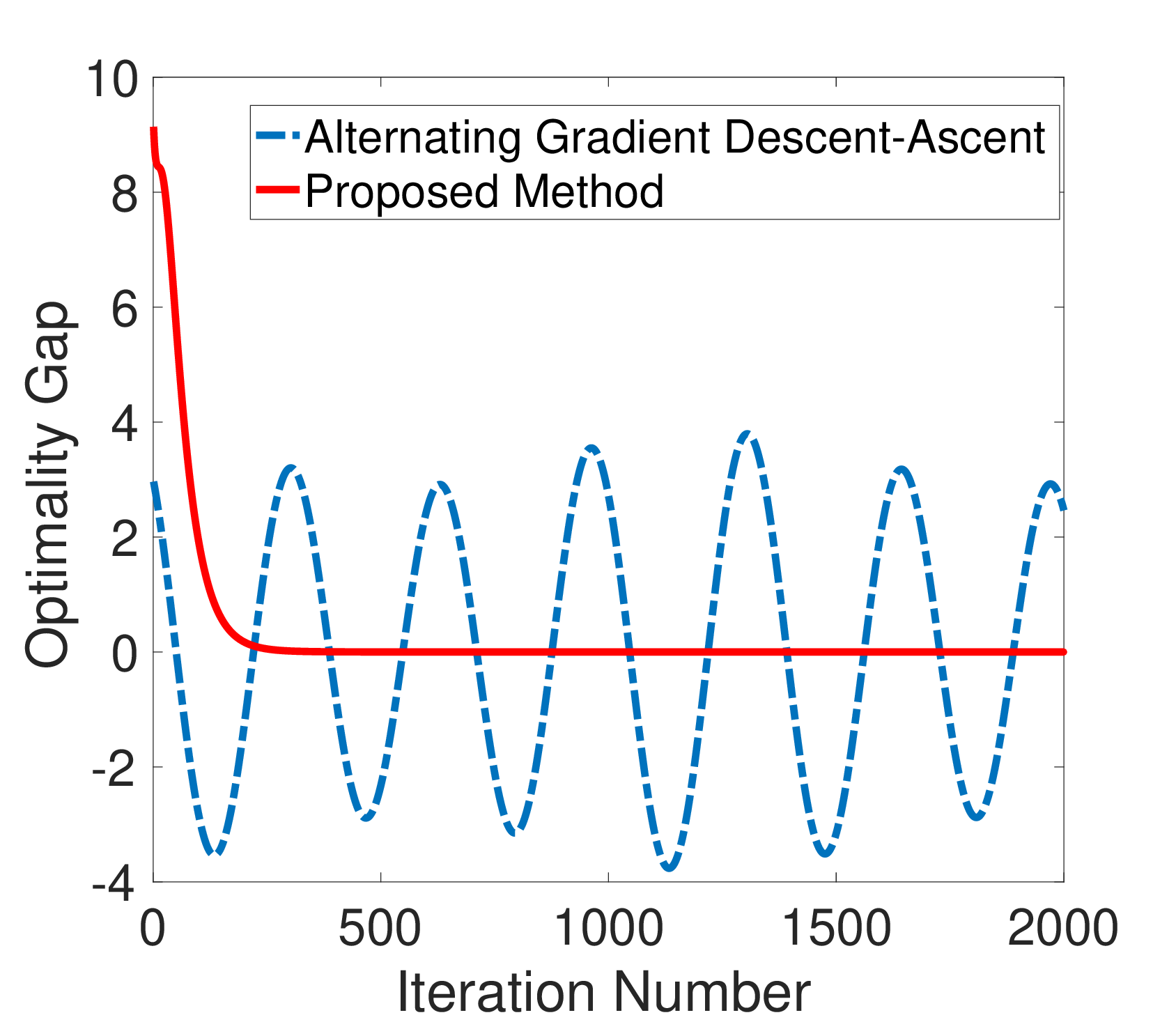}\vspace{-0.1cm}
		\vspace{-0.5cm}
		\caption{\footnotesize Behavior of \eqref{eq:alter:min:max} and Alg. 1. The y-axis is the 1st-order optimality gap $\|Ax\|^2+ \|y^T A\|^2$, which ideally should go to zero. The solid line represents the HiBSA algorithm with $\gamma^r=1/{\sqrt{r}}, \;\beta^r=r, \; \forall~r$. } \label{fig:diverge}
%		\vspace{-0.4cm}
	\end{center}
\end{wrapfigure}

The above example motivates us to introduce {\it both} the proximal term ${\beta^r}/{2}\|x-x^r\|^2$  in {\bf [Step 1]}
of HiBSA, and the penalty term $-\frac{\gamma^r}{2}\|y^r\|^2$ in {\bf [Step 2]}. By properly selecting the sequences $\{\beta^r, \gamma^r\}$, we will show in the next section, that the HiBSA will converge for a wide class of problems (including Example 1 as a special case). 

\section{Theoretical Properties of HiBSA}
\label{sec:typestyle}
We start to present our main convergence results for the HiBSA.
 Our analysis will be divided into three cases according to the structure of  the coupling term $f(x,y)$.  Separately considering different cases of \eqref{eq:problem} is necessary, since the analysis and convergence guarantees could be different. Note that throughout this section, we will assume that $g(y)$ is convex, but {\it not} strongly convex. In case $\ell(x,y)$ is strongly concave in $y$, the strongly concave term will be absorbed into $f(x,y)$.

%\vspace{-0.3cm}
\subsection{Optimality conditions}\label{sub:opt}
First, let us elaborate on the type of solutions we would like to obtain for problem \eqref{eq:problem}. Because of the non-convexity involved in the minimization problem, we will not be able to use the classical measure of optimality for saddle point problems (i.e., the distance to a saddle point). Instead, we will adopt some kind of first-order stationarity conditions. To precisely state our condition, let us define the {\it proximity operator} for $x$ and $y$ blocks as follows:
\begin{align}\label{eq:prox}
&\mbox{\rm Px}^{\beta}_{i}(v_{i}) := \arg\min_{x_{i} \in \mathcal{X}}\;
h_i(x_i) +\frac{\beta}{2}\|x_{i}-v_i\|^2, \; \forall~i\in[K]\\
&\mbox{\rm Py}^{1/\rho}(w) := \arg\max_{y\in \mathcal{Y}}\;  - g(y) - \frac{1}{2\rho}\|y-w\|^2. \nonumber
\end{align}%
Moreover, we define the stationarity gap for problem \eqref{eq:problem} as:{\small
\begin{equation}\label{eq.optgap}
{\nabla} {\mathcal{G}^{\beta}_{\rho}}(x,y):=\left[
\begin{array}{c}\beta(x_1-\mbox{\rm Px}^{\beta}_{1}(x_1-1/\beta\nabla_{x_1}f(x,y)))
\\
\vdots
\\
\beta(x_K-\mbox{\rm Px}^{\beta}_{K}(x_K-1/\beta\nabla_{x_K} f(x,y) ))
\\
1/\rho(y - \mbox{\rm Py}^{1/\rho}(y+ \rho\nabla_y f(x,y)))\end{array}\right].
\end{equation}
}%
We say that a tuple $(x^*,y^*)$ is a first-order stationary solution for problem \eqref{eq:problem} if it holds that:
\begin{align}\label{eq:stationarity}
\|{\nabla}{ \mathcal{G}^{\beta}_{\rho}(x^*,y^*)} \| =0.
\end{align}
To see that \eqref{eq:stationarity} makes sense, first note that if $h\equiv 0, g\equiv 0, \mathcal{Y}=\mathbb{R}^{M}, \mathcal{X}=\mathbb{R}^{NK}$, then it reduces to the condition
$\|[\nabla_x f(x^*,y^*); \nabla_y f(x^*,y^*)]\|=0$, which is {\it independent} of the algorithm parameters $(\beta, \rho)$.}
Further, we  can check that if $y$ is not present, then condition \eqref{eq:stationarity} is equivalent to the first-order stationary condition for the resulting non-convex minimization problem (see, e.g., \cite{bertsekas99}). Further, if $x$  is not present, then condition \eqref{eq:stationarity} simply says that ${ y^{\ast}} \in \arg\max_{y\in Y} \{f(y)-g(y)\}$.

Following the above definition, we will say that $(x^*,y^*)$ is an $\epsilon$-stationary solution if the following holds
\begin{align}\label{eq:epsilon:stationarity}
\|{{\nabla} \mathcal{G}^{\beta}_{\rho}}{(x^*,y^*)}\| \le\epsilon,
\end{align}
As mentioned in the introduction, the stationarity conditions \eqref{eq:stationarity} and \eqref{eq:epsilon:stationarity} are related to the stationarity conditions utilized in \cite{qihang18}, \cite{lin2019gradient}. To illustrate this point, consider the case where $h\equiv 0, g\equiv 0, \mathcal{X}=\mathbb{R}^{NK}$; the stationarity condition \eqref{eq:stationarity} reduces to:
	$$\|\nabla \mathcal{G}(x,y) \|= \|[\nabla_x f(x,y),y-\textrm{proj}_{\mathcal{Y}}(y+\nabla_y f(x,y))]^{T} \|=0.$$ First, consider the case where $f$ is strongly concave in $y$. Then, it can be shown that if $(x^{\ast},y^{\ast})$ is an $\epsilon$-stationary point in the sense of \eqref{eq:epsilon:stationarity}, then it is also an $\mathcal{O}(\epsilon)$ stationary point in the sense defined in \cite{qihang18}, that is:
	\begin{align}
	\|\nabla F(x^{\ast})\| \leq \mathcal{O}(\epsilon), \; \; F(x) := \max_{y \in \mathcal{Y}} f(x,y).
	\end{align}
	Moreover, in the case where $f$ is concave in $y$, for an $\epsilon$-stationary point according to definition \eqref{eq:epsilon:stationarity},  it holds that $\|\nabla F_{1/2 \ell}(x^{\ast})\| \leq \mathcal{O}(\epsilon)$, where $F_{1/2 \ell}$ is the Moreau envelope of $F$ defined as
	\begin{align}
	F_{1/2\ell} (x) = \min_{w} F(w) + {\ell}\|w-x\|^2.
	\end{align} 
	For more details the interested reader can refer to  \cite{lin2019gradient}.

Based on the above definition of first-order stationarity, we establish the equivalence between a few optimization formulations discussed in Section \ref{sec:intro}.
\begin{proposition}
{\it Problems \eqref{eq:distributed} and \eqref{eq:reform} are equivalent, in the sense that every KKT point for problem \eqref{eq:distributed} is a first-order stationary solution of \eqref{eq:reform} [in the sense of \eqref{eq:stationarity}], and vice versa.}
\end{proposition}
\noindent{\bf Proof.}
For simplicity of notation we assume $N=1$. Consider the following KKT conditions for problem \eqref{eq:distributed}
\begin{align}
\begin{split}
&\hspace{-0.1cm}\langle \nabla_x f(x^*) + \xi^*(x^*) \hspace{-0.1cm} +{A}^{T} y^{*},x -x^{*}\rangle \\
& \quad \quad \quad - \left\langle y^{\ast},\tilde{x} - \tilde{x}^{\ast}\right\rangle \geq 0, \forall\; {\rm feasible}\; (x,\tilde{x}) \\
& {A}{x}^*=\tilde{x}^{\ast} \label{eq:opt_distr}
\end{split}
\end{align}
where $\xi^*(x^*) \in \partial h(x^*)$ and $y$ is the Lagrange multiplier.
Now consider a stationary point $(x^{\ast},\tilde{x}^{\ast},y^{\ast})$
of problem \eqref{eq:reform}. Then the stationarity condition \eqref{eq:stationarity} implies that{\small
\begin{subequations}
\begin{align}
x^{\ast} &= \arg\min_{x}\; \left\langle  {A}^{T} y^{*} + \nabla_x f(x^{\ast}), x - x^{\ast} \right\rangle +h({x}) + \frac{\beta}{2}\|x-x^{\ast}\|^2 \label{eq:opt4}\\
\tilde{x}^{\ast}& = \arg\min_{\tilde{x} \in Z}\; \left\langle -y^{\ast}, \tilde{x}- \tilde{x}^{\ast}\right\rangle  + \frac{\beta}{2}\|\tilde{x}-\tilde{x}^{\ast}\|^2 \\
y^{\ast}& = \arg\max_{y}\; \left\langle Ax^{\ast} -\tilde{x}^{\ast}, y- y^{\ast}\right\rangle  - \frac{1}{2\rho}\|y-y^{\ast}\|^2  \label{eq:opt5}
\end{align}
\end{subequations}}%
The optimality conditions for these problems imply
\begin{align}
& {A} ^{T} y^{*} + \nabla_x f(x^{\ast})+ \xi(x^{\ast})=0 \label{eq:opt_distr1}\\
&\left\langle -y^{\ast}, \tilde{x}- \tilde{x}^{\ast}\right\rangle  \geq 0, \; \forall \tilde{x} \in Z , \quad {A}x^{\ast} - \tilde{x}^{\ast}=0 \label{eq:opt_distr3}.
\end{align}
Clearly, the conditions \eqref{eq:opt_distr1} -- \eqref{eq:opt_distr3} imply \eqref{eq:opt_distr}.

Conversely, suppose \eqref{eq:opt_distr} is true.  Set $x = x^{\ast}$ in \eqref{eq:opt_distr} we  obtain condition \eqref{eq:opt_distr3}. Moreover, in order to obtain condition \eqref{eq:opt_distr1} we set $\tilde{x}=\tilde{x}^{\ast}$ in \eqref{eq:opt_distr} and take into account the fact that \eqref{eq:opt_distr} holds $\forall x \in \mathbb{R}^{KN}$. The proof is completed.
\QED

\begin{proposition}
\it	Consider the problem: $\max\limits_{x \in \mathcal{X}} \min\limits_{k}  R_{k}(x)$,  and its reformulation \eqref{eq:max-min-fair}. They are equivalent in the sense that, an equivalent smooth reformulation of the former has the same first-order stationary solutions as those of the latter [in the sense of \eqref{eq:stationarity}].
\end{proposition}
\noindent{\bf Proof.} A well-known equivalent smooth formulation of the min-utility maximization problem is given below (equivalent in that the global optimal of these two problems are the same)
\begin{align}\label{eq:problem:minmax:equiv}
\max_{\lambda, x\in \mathcal{X}}\;\lambda,\quad \mbox{s.t.}\; \; R_k(x)\ge \lambda, {\; \forall k}.
\end{align}
The partial KKT conditions of the above problem are
\begin{align} \label{eq:kkt2}
&\left\langle \sum_{i=1}^{K}\hat{y}_i   \nabla_{x} R_i(\hat{x}), x -\hat{x} \right\rangle \leq 0, \; \forall~x\in \mathcal{X},\\
&{\sum_{i}\hat{y}_i =1}, \; \hat{y}_i\ge 0, \; \hat{y}_i (\hat{\lambda}- R_i(\hat{x}))=0, \; R_i(\hat{x})\ge \hat{\lambda}, \;  \forall~i,\nonumber
\end{align}
where $\{\hat{y}_{i} \}_{i=1}^{K}$ are the respective Lagrange multipliers, which together with $(\hat{\lambda}, \hat{x})$ satisfy the KKT condition.

Now consider a stationary point $(x^{\ast},y^{\ast})$
of problem \eqref{eq:max-min-fair}. Then the optimality conditions \eqref{eq:stationarity} imply that{\small
	\begin{subequations}
	\begin{align}
	\hspace{-0.4cm} x^{\ast} &\hspace{-0.1cm}= \arg\min_{x_i \in \mathcal{X}}\hspace{-0.1cm}\bigg\langle \sum\limits_{i=1}^{K} -y_i^{\ast} \nabla_{x} R_{i}(x^{\ast}) , x - x^{\ast} \bigg\rangle +\frac{\beta}{2}\|x-x^{\ast}\|^2\label{eq:opt2}\\
	\hspace{-0.4cm} y^{\ast}& \hspace{-0.1cm} = \arg\max_{y\in \Delta}\; \left\langle -R(x^{\ast}), y-y^{\ast}\right\rangle  - \frac{1}{2\rho}\|y-y^{\ast}\|^2, \label{eq:opt3}
	\end{align}
	\end{subequations}}%
\hspace{-0.1cm}where $R(x^{\ast}) := [ R_{1}(x^{\ast}); \ldots; R_K(x^{\ast})]$. Let us define 
$$ \lambda^* =  \min\limits_{i=1, \ldots, K} \{ R_i(x^{\ast})\}$$ 
so it holds that $R_i(x^{\ast})\ge  \lambda^*, \;  \forall~i$.

Plugging $(x^*,y^*)$ into the optimality conditions of \eqref{eq:opt2}, \eqref{eq:opt3}, we obtain:
\begin{subequations}
\begin{align}
&\left\langle \sum\limits_{i=1}^{K} -y_i^{\ast} \nabla_{x_i} R_{i}(x^{\ast}) , x - x^{\ast}  \right\rangle \geq 0, \; \forall x \in \mathcal{X} \label{eq:power_opt1}\\
&\langle -R(x^{\ast}) , y-y^{\ast} \rangle \leq 0, \; \forall y \in \Delta, \; y^{\ast} \in \Delta \label{eq:power_opt2}.
\end{align}
\end{subequations}
For all $i$ such that $R_{i}(x^{\ast})= \lambda^*$ obviously it holds that $y_i^* ( \lambda^*- R_i(x^*))=0$. Let $i,j$ be indices such that $R_{i}(x^*)> \lambda^*$ and $R_{j}(x^*)= \lambda^*$. Then, plugging $y_{i}=0,y_{j}=y_{i}^{\ast}+y_{j}^{\ast}$ and $y_{k}=y_{k}^{\ast},k \neq i,j$ into \eqref{eq:power_opt2} yields $y_{i}^{\ast}(R_{j}(x^{\ast})-R_{i}(x^{\ast})) \geq 0$. Because $R_{j}(x^{\ast})-R_{i}(x^{\ast})<0$ and $y_{i}^{\ast} \geq 0$ it must necessarily hold $y_{i}^{\ast}=0$ and thus  $y_{i}^{\ast}( \lambda^*-R_{i}(x^{\ast})) = 0$. As a result the conditions \eqref{eq:kkt2} are satisfied.

Conversely, assume $(x^{\ast},y^{\ast})$ satisfies conditions \eqref{eq:kkt2}. {Note that $R_i(x^*) y_i\ge  \lambda^* y_i$ for any $y\in \Delta$}, so
\begin{align*}
&\langle R(x^{\ast}) , y-y^{\ast} \rangle = \sum\limits_{i=1}^{K} R_i(x^{\ast})(y_{i}-y_{i}^{\ast}) \geq \sum\limits_{i=1}^{K}  \lambda^* (y_{i}-y_{i}^{\ast}) =0,
\end{align*}
for all $y \in \Delta, y^{\ast} \in \Delta$. It is not difficult to see that $(x^{\ast},y^{\ast})$ satisfy the rest of the conditions in \eqref{eq:power_opt1} -- \eqref{eq:power_opt2}. As a result the opposite direction also holds. \QED

%\vspace{-0.4cm}
\subsection{Convergence analysis:  $f(x,y)$ strongly concave in $y$}\label{sub:convergence:1}
Starting this subsection, we will analyze the convergence of HiBSA algorithm. For the ease of presentation, we relegate all the details of the proof to the appendix.

We will first consider a subset of problem \eqref{eq:problem}, where $f(x,y)$ is strongly concave  in $y$. Specifically, we assume the following.

\noindent{\underline{\sf{Assumption C-1}}}. For any $x\in X$, $f(\cdot)$ satisfies the following:
\begin{equation*}
f(x,z)-f(x,y) \le \langle \nabla_{y} f(x,y), z-y\rangle - \frac{\theta}{2} \|z-y\|^2, \; \forall~y,z \in \mathcal{Y},
\end{equation*}
where $\theta>1$ is the strong concavity constant.
{Further assume:
\begin{align}\label{eq:y:approx}
U_y(u; x,y) = \langle\nabla_y f(x, y), u-y\rangle -\frac{1}{2\rho}\|u-y\|^2.
\end{align}
{ where $\rho>0$ is some fixed constant.}
\hfill $\blacksquare$

We note that it can be verified that the jamming problem \eqref{eq:jammer} satisfies Assumption C-1.
Next we will present a series of lemmas which lead to our main result in this subsection. The detailed proof can be found in Appendix Sec. \ref{app:Lemma1} -- \ref{app:T1}.

\begin{lemma}\label{le.descent:x}
(Descent Lemma on $x$)
{\it Suppose that Assumptions A, {B} and C-1 hold. Let $(x^r,y^r)$ be a sequence generated by HiBSA, with $\gamma^r=0$, and $\beta^r=\beta>0, \; \forall~r$.
Then we have the following descent estimate:
\vspace{-2mm}
\begin{align}\notag
\ell(x^{r+1},y^{r})-\ell(x^{r},y^r)\le - \sum_{i=1}^{K}\left(\beta + \mu_i- \frac{L_{x_i}}{2}\right)\|x_i^{r+1}-x_i^r\|^2.
\end{align}}
\end{lemma}

\begin{lemma}\label{le.descent:y}
(Descent Lemma on $y$)
{\it Suppose that Assumptions A, {B} and C-1 hold. Let $(x^r,y^r)$ be a sequence generated by HiBSA, with $\gamma^r=0$, and $\beta^r=\beta>0, \; \forall~r$.
Then we have the following descent estimate:
\begin{align}\notag
&\ell(x^{r+1},y^{r+1})-\ell(x^{r+1},y^r)\le \frac{1}{\rho}\|y^{r+1}-y^r\|^2\\
&\quad -\left({\theta}-(\frac{1}{2\rho}+\frac{\rho L^2_y}{2})\right)\|y^r-y^{r-1}\|^2+\frac{\rho {L^2_y}}{2}\|x^{r+1}-x^r\|^2.\nonumber
\end{align}}
\end{lemma}

\begin{lemma}\label{le.condalg1}
{\it Suppose that Assumptions A, {B} and C-1 hold. Let $(x^r,y^r)$ be a sequence generated by HiBSA, with $\gamma^r=0$, and $\beta^r=\beta>0, \; \forall~r$. Let us define a potential function as {\small
\begin{align}\label{potential_111}
\mathcal{P}^{r+1}\hspace{-0.1cm}:= {\ell(x^{r+1},y^{r+1})}+\bigg(\frac{2}{\rho^2\theta} + {\frac{1}{2\rho}} -4(\frac{1}{\rho}-\frac{L^2_y}{2{\theta}})\bigg)\|y^{r+1}-y^r\|^2.
\end{align}}%
When the following conditions are satisfied:
\begin{equation} \label{eq:case1_cond}
\rho<\frac{\theta}{4L^2_y},\quad \beta >    L^2_y\left(\frac{2}{\theta^2\rho}+ \frac{\rho}{2}\right)  + \frac{L_{x_{i}}}{2} - \mu_{i}, \; \forall i
\end{equation}
then there exist positive constants $c_1,\{c_{2i}\}_{i=1}^{N}$ such that:
\begin{equation}\label{eq.descp}
\mathcal{P}^{r+1}-\mathcal{P}^{r}< -c_1\|y^{r+1}-y^{r}\|^2 - \sum\limits_{i=1}^{K} c_{2i}\|x_{i}^{r+1}-x_{i}^r\|^2.
\end{equation}}
\end{lemma}

Combining the above analysis, we can obtain the following convergence guarantee for the HiBSA algorithm. 
\begin{theorem}\label{th.main}
{\it Suppose that Assumptions A, {B},  C-1 hold. Let $(x^r,y^r)$ be a sequence generated by HiBSA, with $\gamma^r=0$, and $\beta^r=\beta>0, \; \forall~r$, satisfying \eqref{eq:case1_cond}. For a given $\epsilon>0$, let $T(\epsilon)$ denote the first
iteration index, such that  the following holds:
$$T(\epsilon) \:= \min\{r \mid \|{{\nabla} \mathcal{G}^{\beta}_{\rho}}{(x^{r+1},y^{r+1})}\| \le\epsilon, r\ge 1\}.$$
Then, 
{$T(\epsilon) = \mathcal{O}\left(\frac{1}{\epsilon^2} \right)$}}.
\end{theorem}

\subsection{Convergence analysis:  $f(x,y)$  concave in $y$}\label{sub:convergence:3}
\label{conc_case}

Next, we consider the following assumptions for \eqref{eq:problem}.

\noindent{\underline{\sf{ Assumption C-2}}}. Assume that $f(\cdot)$ in \eqref{eq:problem}  satisfies:
\begin{equation*}
f(x,y)-f(x,z) \le \langle \nabla_{y} f(x,z), y-z\rangle, \; \forall~y,z \in \mathcal{Y}, x\in \mathcal{X}.
\end{equation*}
That is, it is {\it concave} in $y$.
{Further, assume that
\begin{align}\label{eq:U:concave}
U_y(u; x,y) = f(x, u) -\frac{1}{2\rho}\|u-y\|^2.
\end{align}
That is, the $y$ update directly maximizes a regularized version of the objective function. Note that  $U_y(u; x;y)$ is strongly concave in $u$, which satisfies the counterpart of Assumption B.1 for $U_y(\cdot)$.
}\hfill $\blacksquare$

{Despite the fact that $f(x,y)$ is no longer strongly concave in $y$, the $y$-update in {\bf{[S2]}} is still relatively easy since it maximizes a strongly concave function. However, the absence of strong {concavity} of $f(x,y)$ in $y$ poses significant challenge in the analysis. In fact, from Example 1 it is clear that directly utilizing the alternating gradient type algorithm may fail to converge to any interesting solutions. Towards resolving this issue, we specialize the HiBSA algorithm, by using a novel {\it diminishing regularization} plus {\it increasing penalty} strategy to regularize the $y$ and $x$ update, respectively (by using a sequence of diminishing $\{\gamma^r\}$, and increasing $\{\beta^r\}$).

We  have the following convergence analysis.  The proofs of the results below can be found in Appendix Sec. \ref{app:lemma4} -- \ref{app:T2}.
\begin{lemma}\label{descentlemma3}
{(Descent lemma) {\it Suppose that Assumptions A, B and C-2 hold. Let $(x^r,y^r)$ be a sequence generated by HiBSA, with {$\gamma^r > 0$} and $\beta^r>L_{x_i}, \; \forall~r,i$. Then we have:}{\small
\begin{align}
&{{\ell}(x^{r+1},y^{r+1})-{\ell}(x^r,y^r)}\le\frac{1}{2\rho}\|y^r-y^{r-1}\|^2\nonumber\\
&\quad -\left({\frac{\beta^{r}}{2} +\mu }-\frac{\rho {L^2_y}}{2}\right)\|x^{r+1}-x^r\|^2-(\frac{\gamma^{r-1}}{2}-\frac{1}{\rho})\|y^{r+1}-y^r\|^2\nonumber
\\
&\quad +\frac{\gamma^r}{2}\|y^{r+1}\|^2-\frac{\gamma^{r-1}}{2}\|y^r\|^2+\frac{\gamma^{r-1}-\gamma^r}{2}\|y^{r+1}\|^2.\label{eq.descentlemma2}
\end{align}}}
\end{lemma}
\vspace{-0.2cm}

Next we show that there exists a potential function, given below, which decreases consistently {\small
\begin{align}
&\mathcal{P}^{r+1}\hspace{-0.1cm}:= \left(\frac{1}{2\rho}+\frac{2}{\rho^2\gamma^r}+\frac{2}{\rho}\left(\frac{1}{\rho\gamma^{r+1}}
-\frac{1}{\rho\gamma^{r}}\right)\right)\|y^{r+1}-y^r\|^2\nonumber\\
&+ {\ell(x^{r+1},y^{r+1})}-\frac{\gamma^r}{2}\|y^{r+1}\|^2-\frac{2}{\rho}\left(\frac{\gamma^{r-1}}{\gamma^r}-1\right)\|y^{r+1}\|^2.\label{eq:potential:3}
\end{align}}
\begin{lemma}\label{potential3}
{\it Suppose that Assumptions A, B and C-2 are satisfied. Let $(x^r,y^r)$ be a sequence generated by HiBSA. Suppose the following conditions are satisfied for all $r$, 
{\small
\begin{equation}\label{assumpt}
\beta^{r} >  \rho L^2_y+\frac{4 L^2_y}{\rho(\gamma^r)^2}- 2\mu, \; \; \beta^{r}>L_{x_i}, \forall i, \; \; \frac{1}{\gamma^{r+1}}-\frac{1}{\gamma^r}\le {\frac{\rho}{5}},
\end{equation}}%
then the change of potential function can be bounded through {\small
\begin{align}\label{eq:whole:2}
\hspace{-0.3cm}&\mathcal{P}^{r+1}\hspace{-0.1cm}\le \mathcal{P}^{r} \hspace{-0.1cm} -\left( {\frac{\beta^r}{2} + \mu} -\left(\frac{\rho {L^2_y}}{2}+\frac{2{L^2_y}}{\rho(\gamma^r)^2}\right)\right)\hspace{-0.1cm}\|x^{r+1}-x^r\|^2
\\
&\hspace{-0.5cm}-\frac{1}{10\rho}\|y^{r+1}-y^r\|^2\hspace{-0.1cm}+\frac{\gamma^{r-1}-\gamma^r}{2}\|y^{r+1}\|^2\hspace{-0.1cm}+\frac{2}{\rho}\left(\frac{\gamma^{r-2}}{\gamma^{r-1}}-\frac{\gamma^{r-1}}{\gamma^r}\right)\|y^r\|^2\nonumber
\end{align}}}
\end{lemma}

Before proving the main result in this section, we make the following assumptions on the parameter choices.

\noindent\underline{\sf{Assumption  C-3.}} Suppose that the following conditions hold:
\begin{enumerate}
\item[(1)] The sequence $\{\gamma^r\}$  satisfies
\begin{align}
\begin{split}
&\gamma^{r}-\gamma^{r+1}\ge 0, \quad {\gamma^{r}}\rightarrow 0, \\
&\sum^{\infty}_{r=1}(\gamma^r)^2=\infty, \quad \frac{1}{\gamma^{r+1}}-\frac{1}{\gamma^r}\le \frac{ \rho}{5}.\label{eq.conditions31}
\end{split}
\end{align}
\item [(2)] The sequence $\beta^r$ satisfies

\begin{align}
\beta^{r} >  \rho L^2_y+\frac{4 L^2_y}{\rho(\gamma^r)^2}- 2\mu, \quad  \beta^{r}>L_{x_i}, \; \forall i.
\label{eq.conditions32}
\end{align}
\end{enumerate}
Note that the above assumption on $\{\gamma^r\}$ can be satisfied, for example, when $\gamma^r = \frac{1}{\rho r^{1/4}}$; see the discussion after \eqref{eq:rate_case3}. \hfill $\blacksquare$

\begin{theorem}\label{th.main3}
{\it Suppose that Assumptions A, B,  C-2 and C-3 hold. Let $(x^r,y^r)$ be a sequence generated by HiBSA. For a given $\epsilon>0$, let $T(\epsilon)$ be defined similarly as in Theorem \ref{th.main}.
Then, 
{$T(\epsilon) = \widetilde{\mathcal{O}}\left(\frac{1}{\epsilon^4} \right).$}
}
\end{theorem}

It is important to note that, when the problem is only concave in $y$, the condition \eqref{eq:U:concave} asserts that in each step a strongly concave problem has to be solved exactly. However, for a generic objective function, this step does not involve closed-form solution.
In the supplementary material accompanying this paper \cite{ieee_minmax_supp}, we extend this algorithm to the case where the maximization problem is solved by performing a finite number of gradient ascent steps.

\vspace{-0.3cm}
\subsection{Convergence analysis:  $f(x,y)$ linear in $y$}\label{sub:convergence:2}

Finally, we briefly discuss the case where the coupling term in \eqref{eq:problem} is linear in $y$. The derivation of the results in this section largely follows from what we have presented in Section \ref{sub:convergence:3}, therefore we choose to omit it.

\noindent{\underline{\sf{Assumption C-4}}}. Assume that problem \eqref{eq:problem} simplifies to:
%\vspace{-5mm}
\begin{align}\label{eq:problem:linear}
\begin{split}
\min_{{x}}\max_{{y}}&\quad y^T F(x_1,x_2,\cdots, x_K) + \sum_{i=1}^{K}h_i(x_i) - g(y)\\
\mbox{s.t.} & \quad x_i\in \mathcal{X}_i, \; y\in \mathcal{Y}, \; i=1,\cdots, K
\end{split}
\end{align}
where $F(\cdot):\mathbb{R}^{NK}\to \mathbb{R}^M$ is a vector function.  Further assume that \eqref{eq:U:concave} holds for $U_y(\cdot)$ \hfill $\blacksquare$

Note that \eqref{eq:problem:linear} contains the robust learning problem \eqref{eq:robust}, the min utility maximization problem \eqref{eq:max-min-fair}, and Example 1 as special cases. It is worth noting that, due to the use of the strongly concave approximation function $U_y(u; x,y) $ as defined in \eqref{eq:U:concave}, we are able to perform a simple gradient step to update $y$, while in the algorithm proposed in the previous section, each iteration has to solve an optimization problem involving $y$.

{It is worth mentioning that, in this case the analysis steps are similar to those in Sec. \ref{conc_case}. In particular, we can show that the potential function \eqref{eq:potential:3} has the same behavior as in Lemma \ref{potential3}. Therefore, we state our convergence result in the following corollary.}

{\begin{corollary}\label{th.main2}
{\it Suppose that Assumptions A, B, C-3 and  C-4 hold. Let $(x^r,y^r)$ be a sequence generated by HiBSA. For a given $\epsilon>0$, let $T(\epsilon)$ be defined  as in Theorem \ref{th.main}.
{Then, $T(\epsilon) = \widetilde{\mathcal{O}}\left(\frac{1}{\epsilon^4} \right).$}
}
\end{corollary}
}

%\vspace{-0.3cm}

\section{Numerical Results}

\label{sec:majhead}
We test our algorithms on three applications: a robust learning problem, a rate maximization problem in the presence of a jammer and {a coordinated beamforming problem}.

{\noindent{\bf Robust learning over multiple domains.} Consider a scenario where we have datasets from two different domains and adopt a neural network  model in order to solve a multi-class classification problem. The neural network consists of two hidden layers with { 50} neurons, each endowed with { sigmoid} activations, except from the output layer where we adopt the softmax activation. We aim to learn the model parameters using the following two approaches:\\

\vspace{-0.2cm}
\noindent [1] \underline{Robust Learning} : Apply the robust learning model (\ref{eq:robust}) and optimize the cost function using the HiBSA algorithm with  $\gamma^{r} = \frac{1}{r^{1/4}}$ { and the Multi-step GDA algorithm \cite{nouiehed2019solving} with one gradient descent and five gradient ascent steps per iteration}. Note that we treat the minimization variable as one block and use the first-order Taylor expansion of the cost function as the approximation function.
\\

\vspace{-0.2cm}
\noindent [2] \underline{Mutltitask Learning} : Apply a multitask learning model \cite{DBLP:journals/corr/ZhangY17aa}, where  we optimize the sum of the respective empirical risks correspsonding to the two domains/tasks; the weights associated with each task are fixed to 1/2. The problem is optimized using gradient descent.\\

\vspace{-0.2cm}
Moreover, we evaluate the above algorithms by using the {minimum accuracy} across the two domains, over both training and test datasets. That is, accuracy = $\min\lbrace${accuracy} on domain 1, {accuracy} on domain 2$\rbrace$. 
}

{In our experiments we use the MNIST \cite{lecun1998gradient} dataset whose data points are images of handwritten digits of dimensions $28 \times 28$. We select two different parts of the MNIST dataset as the two different domains we mentioned above. The first part consists of the digits from 0 to 4, while the second one contains the rest. Moreover, for the first domain we use $5000$ images for training and $1000$ for testing, while in the second one we employ $25000$ and $5000$ images respectively.  Finally, we average the results over $5$ iterations.

Note that we do not perform  extensive parameter tuning, since the purpose of this experiment is not to support the superiority of the robust model, but merely to illustrate that the proposed HiBSA computes a reasonable model similar to what can be computed by multistep GDA, and to what can be obtained by multi-task learning. 
Indeed, the results presented in Fig. \ref{fig:mnist} support this view, since different approaches achieve approximately the same accuracy on the test set.

}

\begin{figure}
		\vspace{-0.5cm}
\begin{center}
\begin{minipage}[t]{0.35\textwidth}
\includegraphics[width=\textwidth]{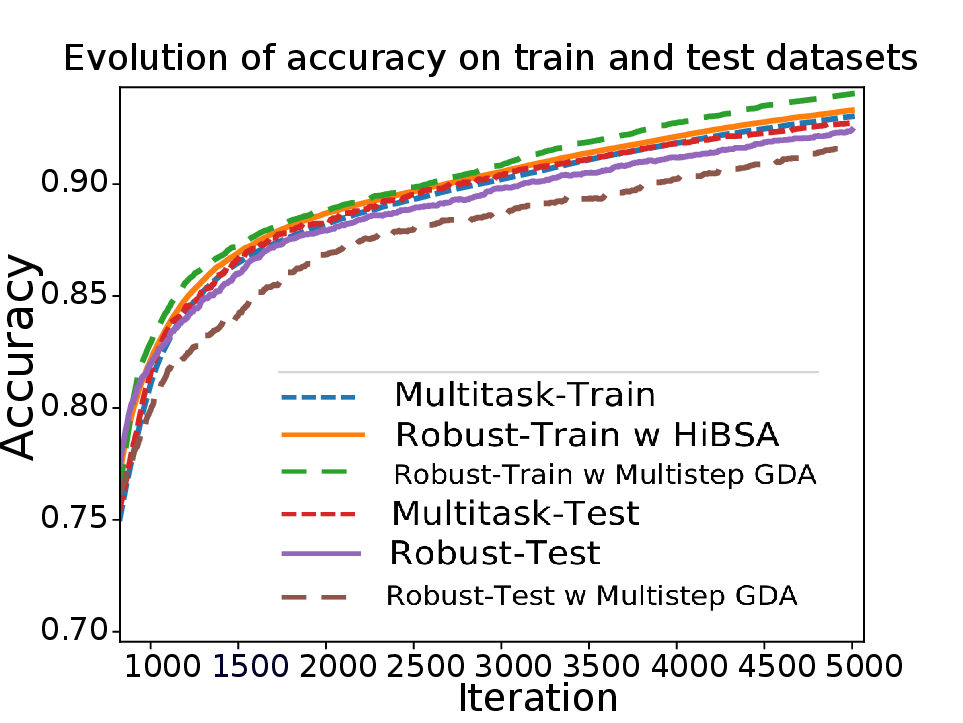}
\end{minipage}
\begin{minipage}[t]{0.35\textwidth}
\includegraphics[width=\textwidth]{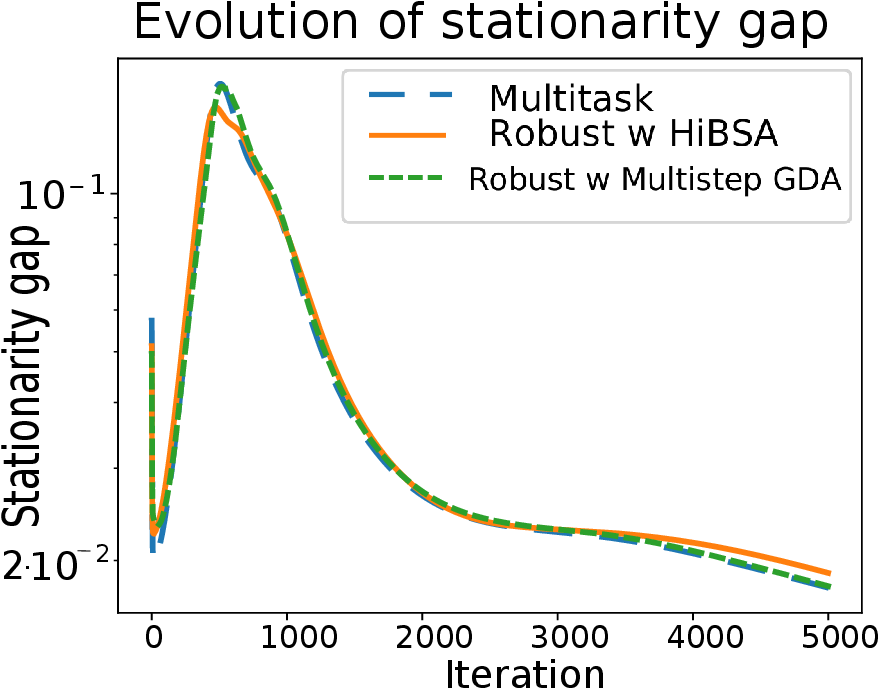}
\end{minipage}
\end{center}
	\vspace{-0.5cm}
\caption{\footnotesize {The results on the experiments performed on the MNIST dataset \cite{lecun1998gradient}. The top figure depicts training and testing accuracies, while the second one  depicts the convergence behavior of the two algorithms.}}
\vspace{-0.3cm}
\label{fig:mnist}

\end{figure}
\noindent{\bf Power control in the presence of a jammer.} Consider the multi-channel and multi-user  formulation \eqref{eq:jammer} where there are $N$ channels, $K$ collaborative users and one jammer. We can verify that the jammer problem (i.e., the maximization problem over $y$) has a strongly concave objective function over the feasible set.

We compare HiBSA with the classic interference pricing algorithm \cite{Schmidt13compare,Shi:2009}, and the WMMSE algorithm \cite{shi11WMMSE_TSP}, which are designed for solving sum-rate optimization problem {\it without} the jammer. Our problem is tested using the following setting. We construct a network with $K=10$, and the interference channel among the users and the jammer is generated using  uncorrelated fading channel model with channel coefficients generated from the complex zero-mean Gaussian distribution with unit covariance \cite{shi11WMMSE_TSP}. All  users' power budget is fixed at $P=10^{\rm SNR/10}$. For test cases without a jammer, we set $\sigma^2_k=1$ for all $k$. For test cases with a jammer, we set $\sigma^2_k=1/2$ for all $k$, and let the jammer have the rest of the noise power, i.e., $p_{0,\max}= N/2$. Note that by splitting the noise power we intend to achieve some  fair comparison between the cases with and without the jammer. However, it is not possible to be completely fair because even though the total noise budgets are the same, the noise power transmitted by the jammer has to go through the random channel, so the total received noise power could  be different. Nevertheless, this setting is sufficient to demonstrate the behavior of the HiBSA algorithm.

From the Fig. \ref{fig:jamming} (top),   it is clear that the pricing algorithm monotonically increases the sum rate (as is predicted by theory), while HiBSA behaves differently: after some initial oscillation, the algorithm converges to a value that has a lower sum-rate. Further in Fig. \ref{fig:jamming} (bottom), we do see that by using the proposed algorithm, the jammer is able to effectively reduce the total sum rate of the system.
\begin{figure}[t]
\begin{center}
		\vspace{-1cm}
\begin{minipage}[t]{0.35\textwidth}
\includegraphics[width=\textwidth]{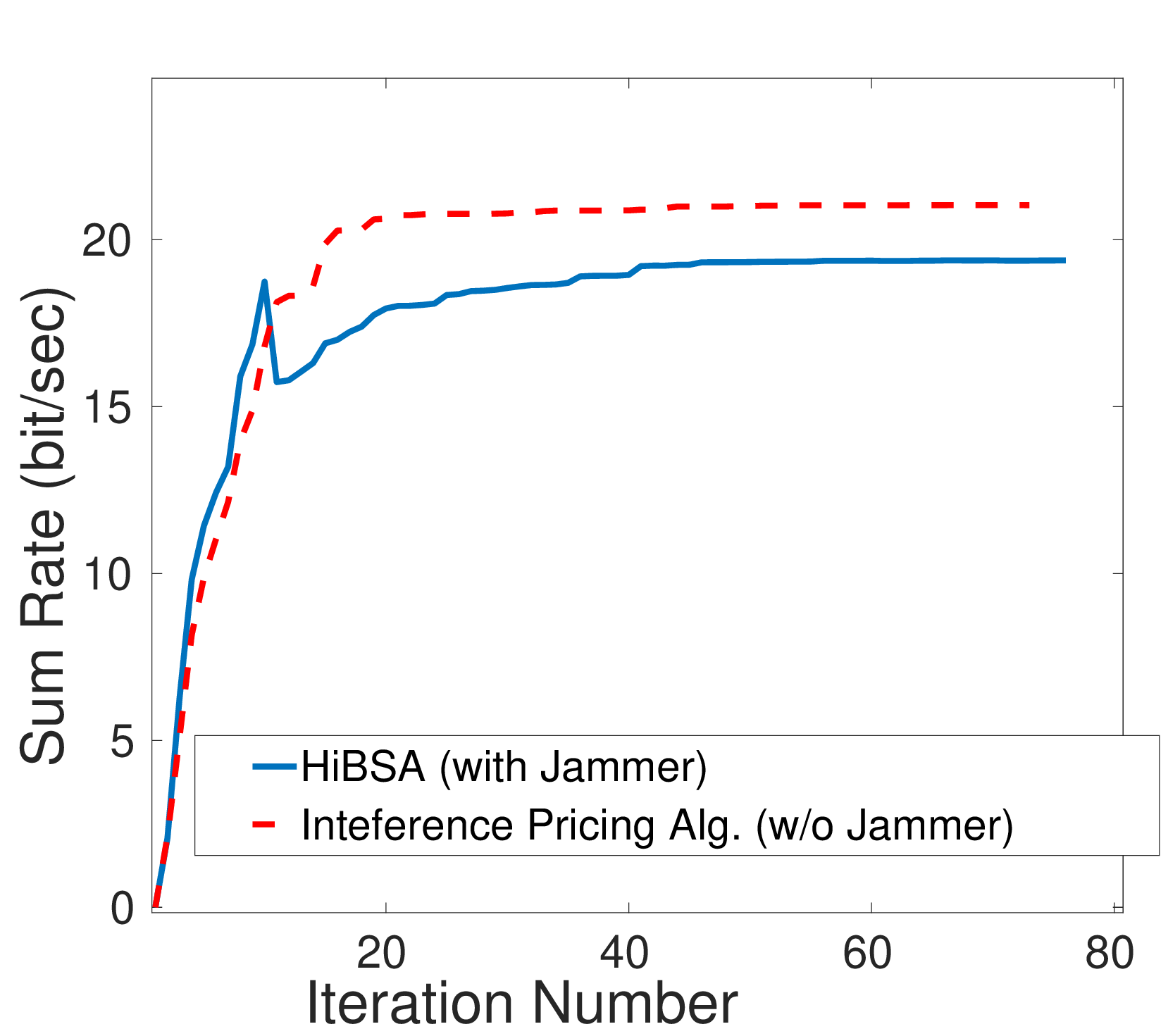}
\end{minipage}
\begin{minipage}[t]{0.35\textwidth}
\includegraphics[width=\textwidth]{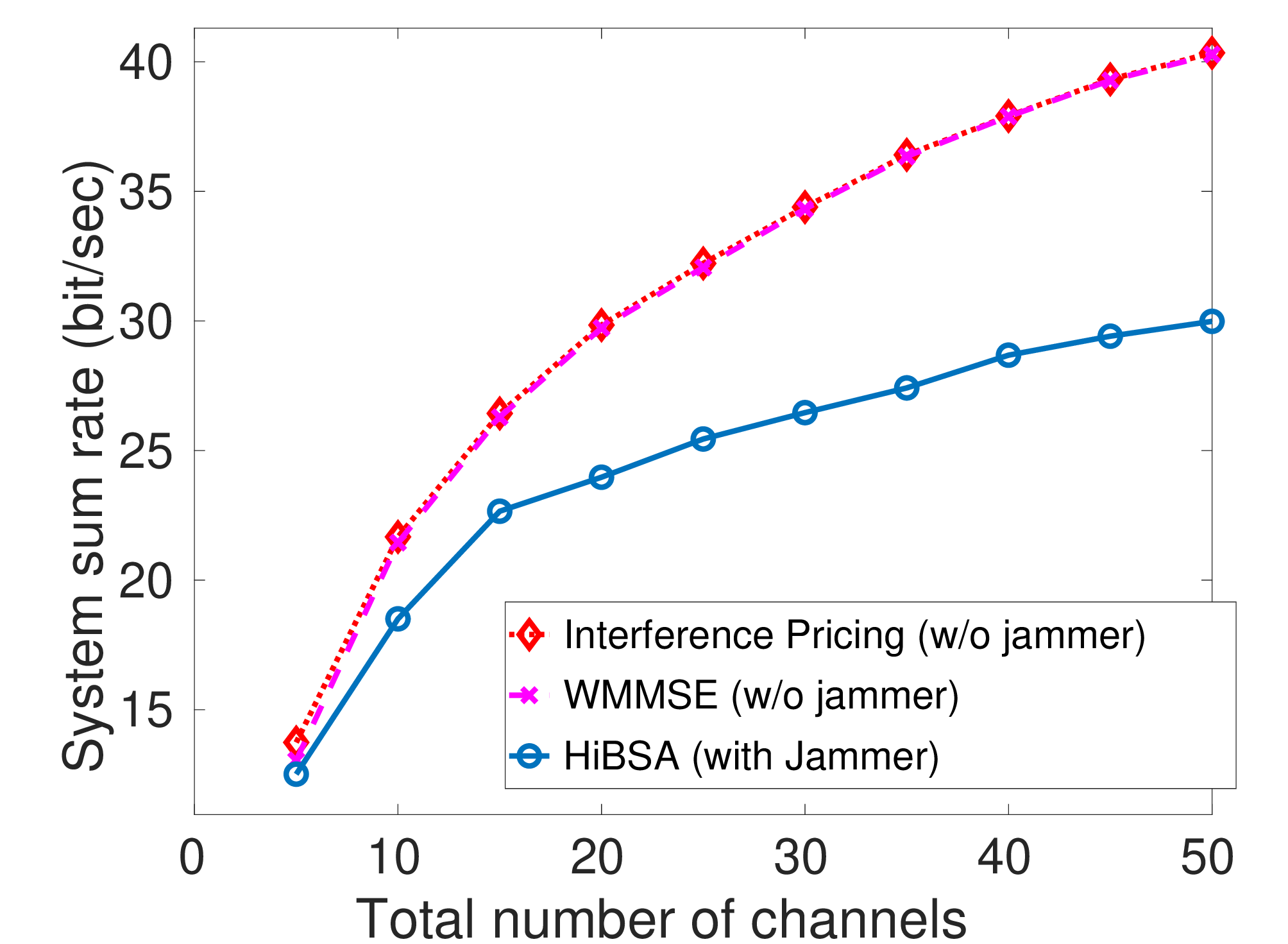}
\end{minipage}
\end{center}
\vspace{-0.4cm}
\caption{\footnotesize The convergence curves and total averaged system performance comparing three algorithms: WMMSE, Interference Pricing and HiBSA. The first figure shows a single realization of the algorithms, and in the second figure, each point represents an average of $50$ realizations. The total number of user is $10$, and ${\rm SNR}=1$.} 
\label{fig:jamming}
\vspace{-0.5cm}
\end{figure}

{\noindent{\bf Coordinated MISO beamforming design.}
Consider the coordinated beamforming design problem \cite{cobf} described in Sec. \ref{sec:intro} over a MISO interference channel. In this problem we experiment with the scenario where there are $K=10$ transmitter-receiver pairs, each transmitter is equipped with  $N=6$ antennas. We adopt the min-rate utility, i.e., $U(\{R_{i}(x)\}_{i=1}^{K}) = \min_{i=1, \ldots, K} \{R_{i}(x)\}$. Moreover, the transmission is performed over a complex Gaussian channel, and we set the power budget to be $\bar{p}=1$. The channel covariance matrices $\{C_{ij}\}, i,j=1,\ldots,10$ are generated at random and their maximum eigenvalues are normalized to 1, if $i=j$, and to some constant $\lambda>0$, if $i \neq j$. Thus, the parameter $\lambda$ quantifies the level of intereference.

The problem of interest is to design the users' beamformers in order to maximize the system's utility function under constraints in power and outage probability. We approach the solution of the problem using two different algorithms :

\noindent[1]\underline{BSUM-LSE \cite{cobf}}:
Substitute the min-rate utility function with  a popular log-sum-exp approximation, i.e., 
$$\min\limits_{i=1,\ldots,K} \{R_i(x)\}:= r_{\min}(x) \approx \frac{1}{\nu}\log_2 \left(\sum\limits_{i=1}^{K} 2^{-\nu R_i} \right).$$ Note that $\nu$ specifies the level of approximation with higher $\nu$'s corresponding to tighter bounds for the approximation error. Then following what is suggested in \cite[Section C]{cobf}, we formulate the respective problem using the surrogate function, and solve the resulting problem iteratively using the projected gradient descent.

\noindent[2]\underline{HiBSA}: We apply the HiBSA to solve the formulation in \eqref{eq:max-min-fair}. The $x$-subproblem is solved similarly as in BSUM-LSE. Moreover, in the maximization problem we use  $\gamma^{r}=1/r^{1/4}$.
	\begin{figure}
		\begin{center}
			\begin{minipage}[t]{0.35\textwidth}
				\centering
				\includegraphics[width=\textwidth]{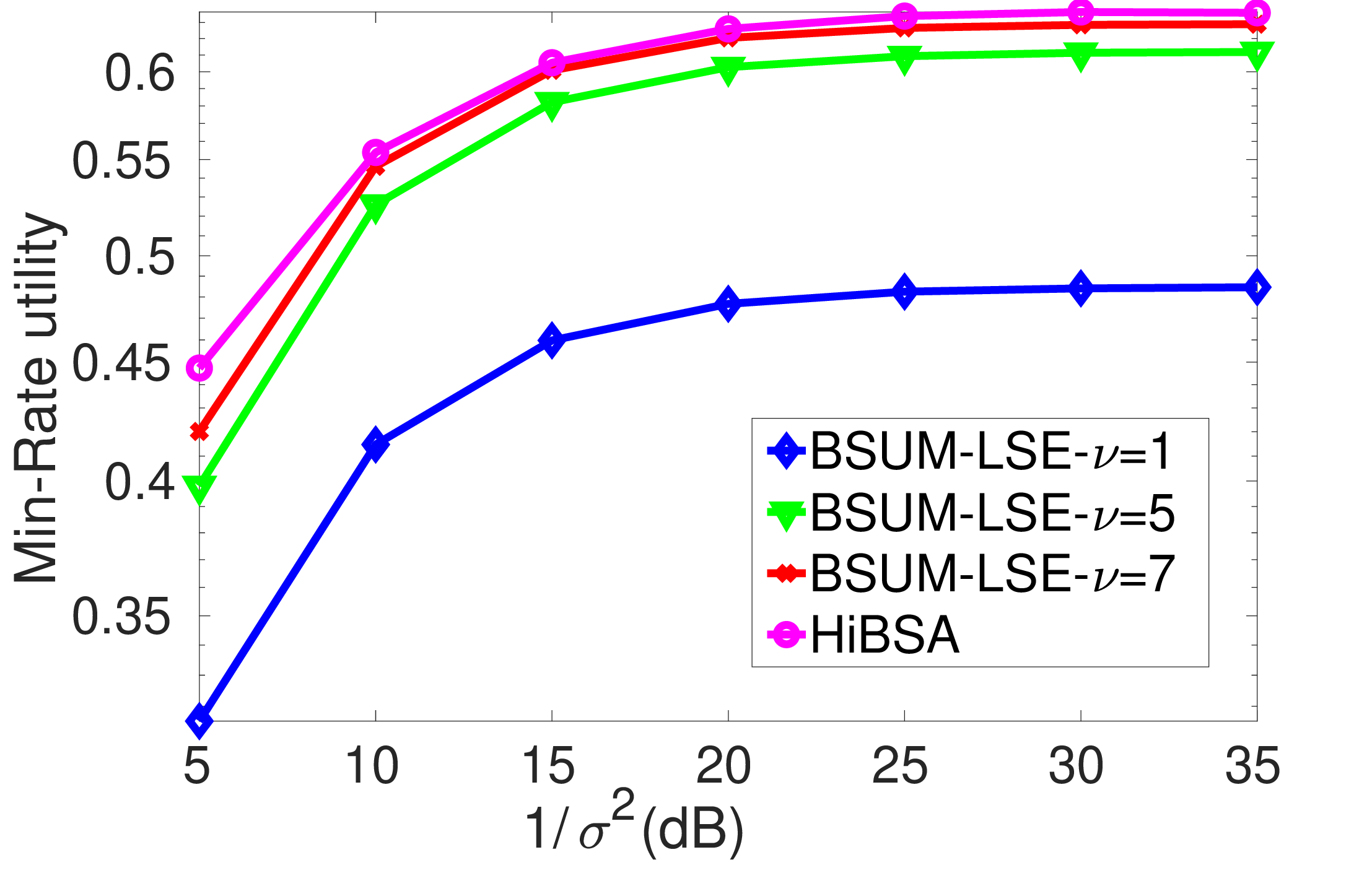}
			\end{minipage}
			\begin{minipage}[t]{0.35\textwidth}
				\centering
				\includegraphics[width=\textwidth]{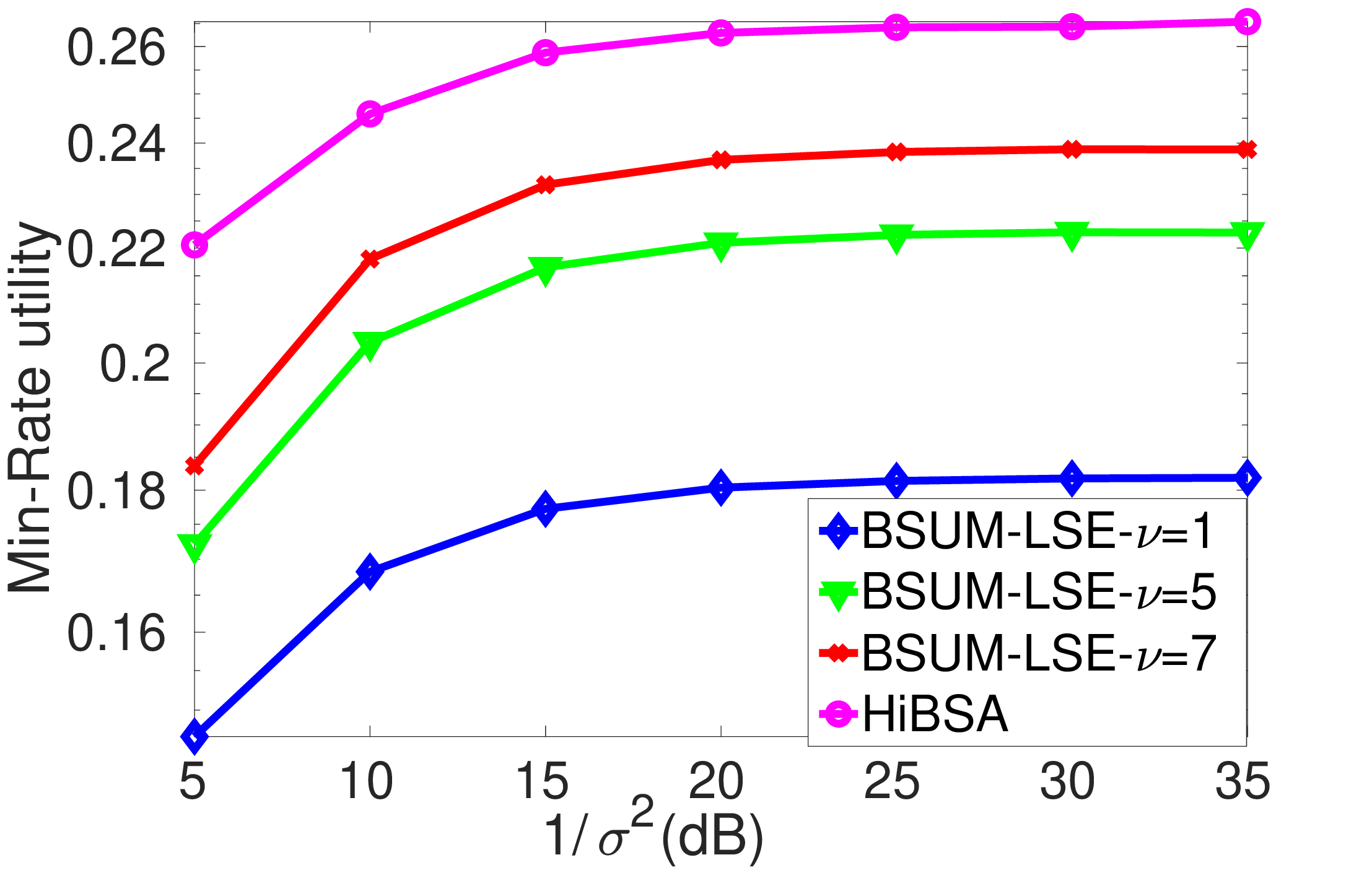}
			\end{minipage}
			\caption{\footnotesize {The min-rate utility achieved using the HiBSA and the BSUM-LSE algorithm \cite{cobf} for two different interference levels {in a scenario where we have $K=10$ users equipped with $N=6$ antennas}. Note that the top figure corresponds to a lower interference level ($\lambda = 0.6$) than the bottom one ($\lambda = 1$). For each interference level we experiment with 3 different values of the level of approximation ($\nu$). Finally, $\sigma^2$ is the noise power at the receivers.}}
			\label{fig:beam}
		\end{center}
			\vspace{-0.8cm}
	\end{figure}}
%	\vspace{-1mm}
\noindent We run both algorithms for $1000$ complete iterations  (one complete iteration involves one update of {\it all the block variables}, and $1000$ iterations are sufficient for both algorithms to converge in all scenarios), set the stepsizes {$\beta$ and $\rho$} {of HiBSA and the respective {stepsize} of BSUM-LSE} all equal to $10^{-2}$. 
We also average the final results over $10$ independent random problem instances. Moreover, in order to evaluate the effect of the log-sum approximation we show the achieved min-rate utility of BSUM-LSE, by using  3 different values of $\nu\in\{1, 5, 7\}$.

 In Fig. \ref{fig:beam} we plot the min-rate utility for  7 different values of the noise variance and 2 different levels of interference. Notice that the HiBSA algorithm achieves higher utility than BSUM-LSE, while as expected the larger the value of $\nu$ the higher the utility achieved by the latter algorithm.

{Furthermore, since large values of the parameter $\nu$ lead to low approximation error bounds, it is of interest to consider experiments with large $\nu$ for the BSUM-LSE algorithm. Intuitively, we expect the resulting objective to be very close to the min-rate utility and thus the achieved min-rate of the BSUM-LSE algorithm should approach the respective min-rate of HiBSA. In order to determine the behavior of BSUM-LSE in that range of $\nu$'s we consider an experiment with $K=10, N=6$ and $\lambda=0.6$. Regarding the stepsizes we keep them constant across the different values of $1/\sigma^2$, however an effort was made to select the optimal ones for all algorithms in order to ensure fair comparisons. Moreover, we terminate both algorithms when the relative successive differences of the min-rate utility becomes small, i.e., $|r_{min}(x^{r+1})-r_{min}(x^r)|/|r_{min}(x^r)|\le 10^{-7}$, or the number of iterations becomes larger than $5,000$. Finally, the results are provided in Fig. \ref{fig:beam2}. 

Note that, for large value of  $\nu$, i.e. $\nu=50, 100, 1000$,  the achieved rate of BSUM-LSE is close but still inferior to that of HiBSA. Additionally, for the same $\nu$'s the HiBSA is faster than BSUM-LSE; in fact the larger the $\nu$ the longer the runtime. On the other hand, the former algorithm is in general slower than BSUM-LSE with $\nu=5$, however in that case HiBSA achieves higher min-rate utility. Overall, note that even though large  $\nu$ leads (in most cases), to improvements in the attained min-rate utility, it also incurs longer runtimes. This can be attributed to the fact that for high $\nu$ the log-sum-exp objective approaches a non-smooth function, which is difficult to optimize. In conclusion, HiBSA in general outperforms BSUM-LSE in terms of runtime and attained min-rate utility.
}
\begin{figure}
		\begin{center}
			\begin{minipage}[t]{0.35\textwidth}
				\centering
				\includegraphics[width=\textwidth]{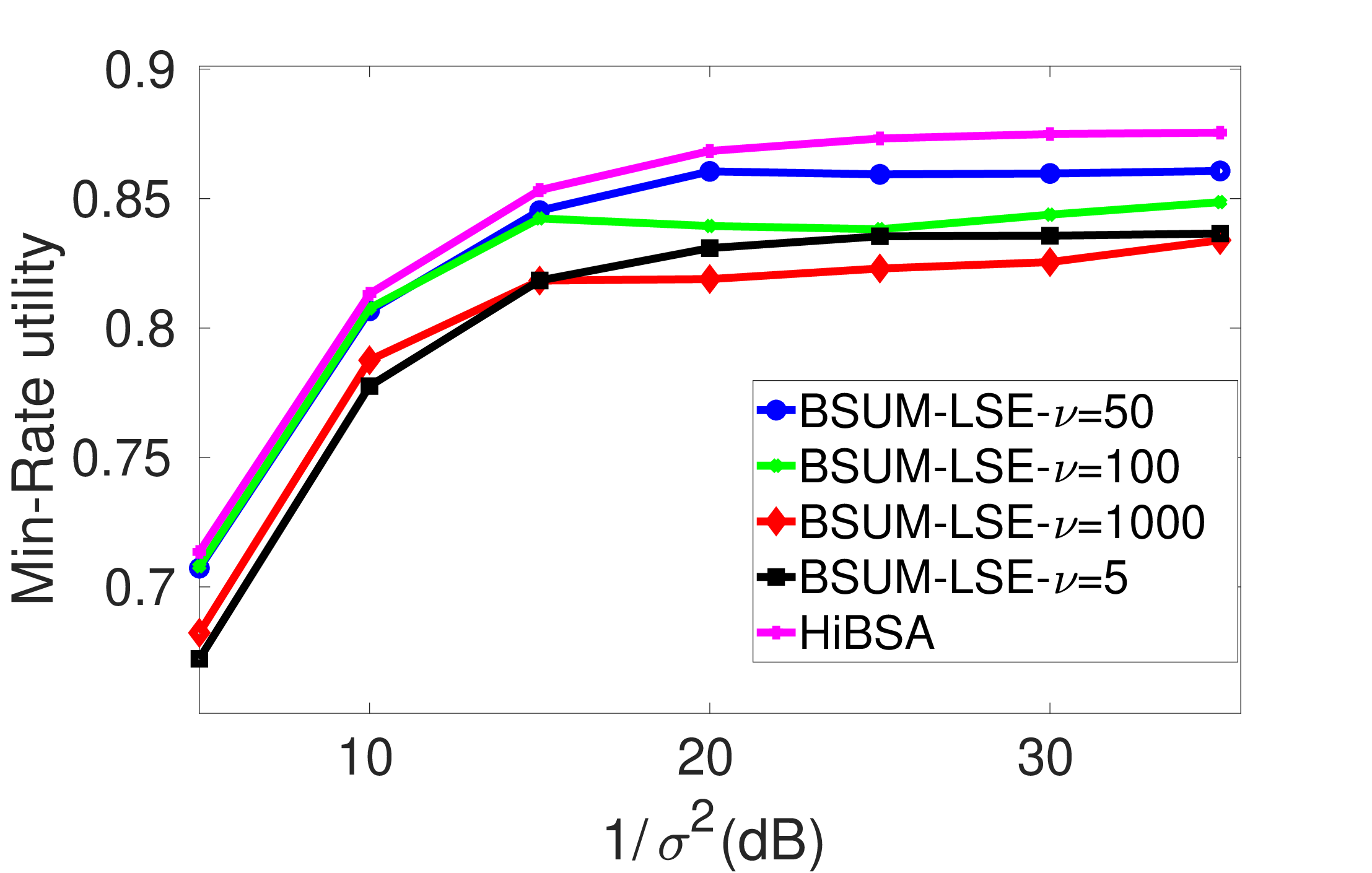}
			\end{minipage}
			\begin{minipage}[t]{0.35\textwidth}
				\centering
				\includegraphics[width=\textwidth]{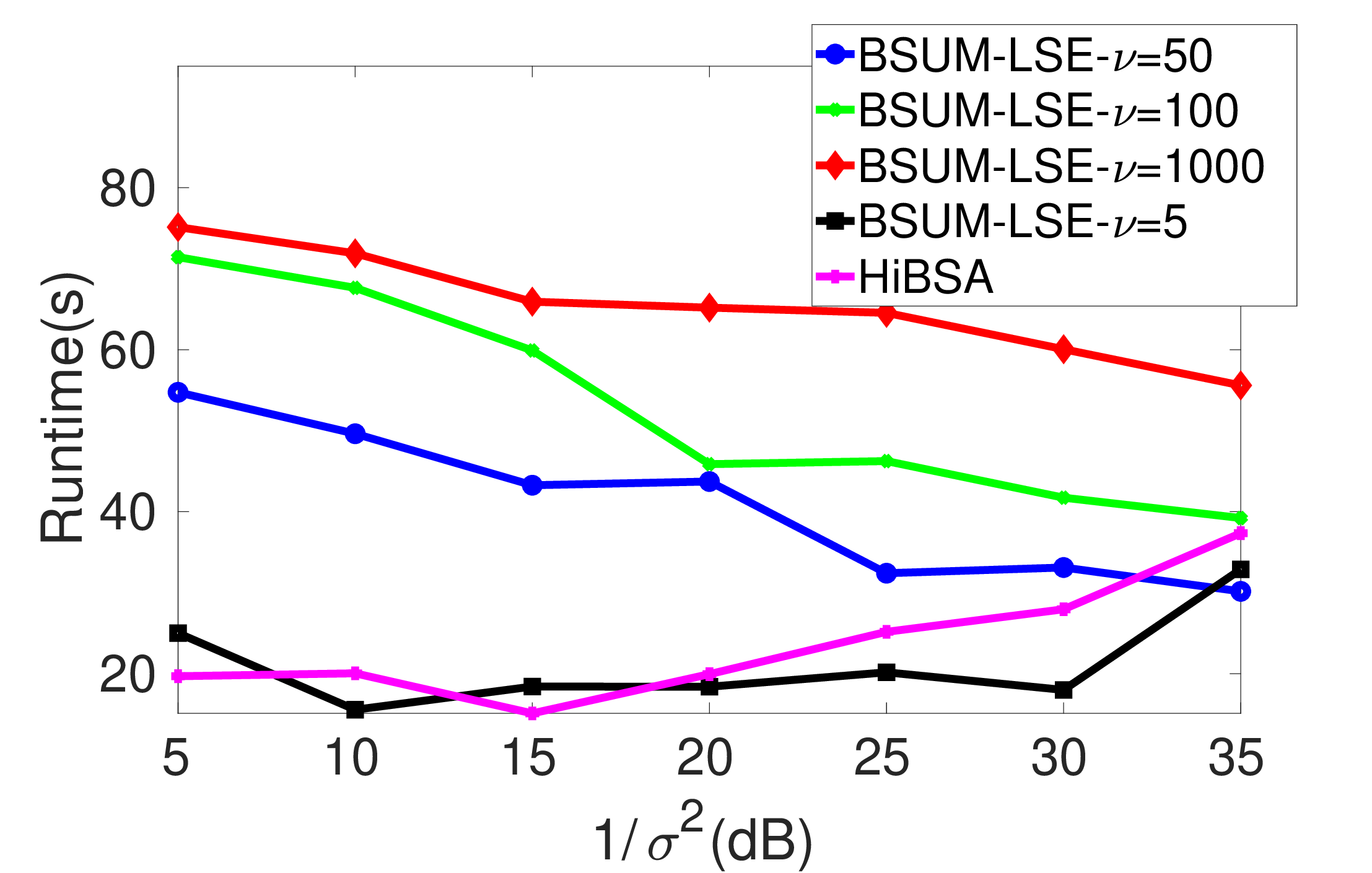}
			\end{minipage}
			\caption{\footnotesize {The achieved min-rate utility (top) and runtime (bottom) for the BSUM-LSE algorithm, with $\nu=5, 50, 100, 1000$, and the HiBSA algorithm. The results are averaged over 30 independent runs.}}
			\label{fig:beam2}
		\end{center}
			\vspace{-0.8cm}
	\end{figure}}

\section{Conclusions}
\label{ssec:subhead}

In this paper, motivated by the min-max problems arising
in the areas of signal processing and wireless communications,
we propose an algorithm called HiBSA. By
leveraging the (strong) concavity of the maximization problem,
we conduct analysis on the convergence behavior of the
proposed algorithm. Numerical results show the effectiveness
of the proposed algorithms of solving the min-max problems
in robust machine learning and wireless communications.  There are many potential future research directions we plan to explore. For example, it will be interesting to develop algorithms for more challenging problems where the $y$ problem is also non-convex. Further, it will also be interesting to establish some lower complexity bounds for non-convex and/or non-concave min-max problems, which characterizes the {\it best} performance one can achieve when optimizing such a family of problems. 

\section{Acknowledgement}
The authors would like to thank Dr. Tsung-Hui Chang and Dr. Wei-Chiang Li for helpful discussion on the MISO beamforming problem, and for providing their codes.

\vspace{-2mm}

\bibliographystyle{IEEEbib}
\bibliography{ref,refs,distributed_opt,ref_tiancong,Surbib}
%\vspace{-0.6cm}
{ \appendix \label{app}
%\vspace{-2mm}

\subsection{Proof of Lemma \ref{le.descent:x}}\label{app:Lemma1}
{By using the assumption that $f$ has Lipschitz gradient, $h_i$ is convex (cf. Assumption A), and by noticing that $w^{r+1}_i=(x^r_i, w^{r+1}_{-i})$, we obtain the following:
\begin{align}\label{eq:desc_x}
&{\ell}(x_{i}^{r+1}, w_{-i}^{r+1}, y^{r}) -
{\ell}(x_{i}^{r}, w_{-i}^{r+1}, y^{r}) \\
&\leq \langle \nabla_{x_{i}} f(w_{i}^{r+1}, y^{r}) + \vartheta^{r+1}_i, x_{i}^{r+1} - x_{i}^{r} \rangle + \frac{L_{x_{i}}}{2} \| x_{i}^{r+1} - x_{i}^{r} \|^{2}\nonumber
\end{align}
for some $\vartheta^{r+1}_i\in \partial h_i(x^{r+1}_i)$.

Second,  the optimality condition for {the} $x_i$ update step  \eqref{eq:x-update} is
\begin{align}\label{eq:opt_xi}
&\langle \nabla_{x_{i}} U_{i}(x_{i}^{r+1}; w_{i}^{r+1}, y^{r})  \nonumber \\
& \quad + \vartheta_i^{r+1} + \beta(x_{i}^{r+1} - x_{i}^{r}), x_{i}^{r} - x_{i}^{r+1} \rangle \geq 0.
\end{align}
So adding and subtracting  $\langle \nabla_{x_{i}} U_{i}(x_{i}^{r}; w_{i}^{r+1}, y^{r}), x_{i}^{r} - x_{i}^{r+1}\rangle $ in \eqref{eq:opt_xi}, and by applying assumptions B.1 (strong convexity) and B.2 (gradient consistency), we obtain the following:
\begin{align*}
&\langle \nabla_{x_{i}} f(w_{i}^{r+1}, y^{r}), x_{i}^{r+1} - x_{i}^{r} \rangle + \langle \vartheta^{r+1}_i, x_{i}^{r+1} - x_{i}^{r} \rangle \nonumber\\
&\leq
-\mu_{i} \| x_{i}^{r+1} - x_{i}^{r} \|^{2} - \beta  \|x_{i}^{r+1} - x_{i}^{r}\|^{2}.
\end{align*}
Then, combining the above expression with \eqref{eq:desc_x} results in 
\begin{align*}
&\hspace{-0.1cm}{\ell}(x_{i}^{r+1}, w_{-i}^{r+1}, y^{r}) - {\ell}(x_{i}^{r}, w_{-i}^{r+1}, y^{r}) \nonumber\\
&\leq \hspace{-0.1cm} \big(\hspace{-0.1cm} -\beta -\mu_{i} + \frac{L_{x_{i}}}{2}\big) \hspace{-0.05cm} \| x_{i}^{r+1} - x_{i}^{r} \|^{2}.
\end{align*}
Summing over $i\in [K]$ we obtain the desired result.  \QED

\subsection{Proof of Lemma \ref{le.descent:y}}\label{app:Lemma2}
For notational simplicity, let us define
$$\ell'(x^{r+1},y)=f(x^{r+1},y)+ {\sum\limits_{i=1}^{K} h_i(x_i^{r+1})}-\mathcal{I}_{\mathcal{Y}}(y) -g(y).$$
Notice that for any $y\in\mathcal{Y}$, we have {$\ell'(x^{r+1},y) = \ell(x^{r+1},y)$}.
The optimality condition of the $y$-step \eqref{eq:y-update} becomes
\begin{align}\label{eq:opt:y}
0&=\nabla_y f(x^{r+1},y^r)-\frac{1}{\rho}(y^{r+1}-y^r)- \xi^{r+1},
\end{align}
where $\xi^{r+1}\in \partial (\mathcal{I}_{\mathcal{Y}}(y^{r+1})+ g(y^{r+1}))$ is a subgradient vector.
Since $\ell'(x,y)$ is {concave} with respect to $y$, we have{\small
\begin{align}
\notag
&\ell'(x^{r+1},y^{r+1})-\ell'(x^{r+1},y^r)\le  \langle\nabla_y f(x^{r+1},y^r) -\xi^r,y^{r+1}-y^r\rangle
\\\notag
\stackrel{\eqref{eq:opt:y}}=&  \frac{1}{\rho}\|y^{r+1}-y^r\|^2 {+\langle\xi^{r+1}-\xi^{r}},y^{r+1}-y^r\rangle
\\\notag
\stackrel{(a)}= &\frac{1}{\rho} \|y^{r+1}-y^r\|^2+\langle\nabla_y f(x^{r+1},y^{r})-\nabla_y f(x^r,y^{r-1}),y^{r+1}-y^r\rangle
\\
&\quad-\frac{1}{\rho}\langle y^{r+1}-y^r-(y^r-y^{r-1}),y^{r+1}-y^r\rangle \nonumber
\\\notag
\mathop{=}\limits^{(b)}&\frac{1}{\rho}\|y^{r+1}-y^r\|^2+\langle\nabla_y f(x^{r+1},y^r)-\nabla_y f(x^r,y^{r-1}),y^{r+1}-y^r\rangle
\\\notag
&
+\frac{1}{2\rho}\|y^r-y^{r-1}\|^2-\frac{1}{2\rho}\|y^{r+1}-y^r\|^2-\frac{1}{2\rho}\|v^{r+1}\|^2
\\\notag
\mathop{\le}\limits^{(c)}&\frac{1}{\rho}\|y^{r+1}-y^r\|^2+\frac{\rho {L^2_y}}{2}\|x^{r+1}-x^r\|^2+\frac{1}{2\rho}\|y^r-y^{r-1}\|^2
\\
&-\frac{1}{2\rho}\|v^{r+1}\|^2 +\langle\nabla_y f(x^r,y^r)-\nabla_y f(x^r,y^{r-1}),y^{r+1}-y^r\rangle\label{eq.end}
\end{align}}%
where $(a)$  follows from the optimality conditions of the $y$-step \eqref{eq:y-update} at iterations $r+1$ and $r$; in $(b)$ we apply the following identity: {\small
\begin{equation}
\langle v^{r+1}, y^{r+1}-y^r\rangle=\frac{1}{2}\left(\|y^{r+1}-y^r\|^2+\|v^{r+1}\|^2-\|y^r-y^{r-1}\|^2\right)\label{eq.eqrel},
\end{equation}}% 
where we have defined 
\begin{align}\label{eq:v}
v^{r+1}:=y^{r+1}-y^r-(y^r-y^{r-1});
\end{align}
in $(c)$ we add and subtract a term $\langle \nabla_y f(x^r,y^r), y^{r+1}-y^r\rangle$, and apply the Young's inequality and obtain:
{\begin{align}
&\langle\nabla_y f(x^{r+1},y^r)-\nabla_y f(x^r,y^r),y^{r+1}-y^r\rangle\nonumber\\
&\le\frac{\rho {L^2_y}}{2}\|x^{r+1}-x^r\|^2+\frac{1}{2\rho}\|y^{r+1}-y^r\|^2,
\label{eq.lip}
\end{align}}%
{where ${L_y}$ is defined in \eqref{eq:Ly}.}
By applying the strong concavity of $f(x,y)$ in $y$, the Young's inequality {and the Lipschitz condition w.r.t $y$}, we can have the following bound for the inner product term in \eqref{eq.end}:
\begin{align}
&\langle\nabla_y f(x^r,y^r)-\nabla_y f(x^r,y^{r-1}),y^{r+1}-y^r\rangle\nonumber\\
&\le\langle\nabla_y f(x^r,y^r)-\nabla_y f(x^r,y^{r-1}),v^{r+1}+y^r-y^{r-1}\rangle
\nonumber \\
&\le\frac{\rho L^2_y}{2}\|y^r-y^{r-1}\|^2+\frac{1}{2\rho}\|v^{r+1}\|^2-\theta\|y^r-y^{r-1}\|^2. \label{eq.lip1}
\end{align}
Combining the above with \eqref{eq.end} completes the proof.  \QED

At this point, by simply combining Lemmas \ref{le.descent:x} - \ref{le.descent:y}, it is not clear how the objective value behaves after each $x$ and $y$ update.
To capture the essential
dynamics of the algorithm, the key is to identify a proper potential function, which decreases after each round of $x$ and $y$ updates.

\subsection{Proof of Lemma \ref{le.condalg1}}\label{app:Lemma3}
According to \eqref{eq:opt:y}, the optimality condition of $y$-problem \eqref{eq:y-update} at iterations $r+1$ and $r$ are given by:
\begin{align}
&- \nabla_y f(x^{r+1},y^r) +\xi^{r+1} + \frac{1}{\rho}(y^{r+1}-y^r)=0,
\\
&- \nabla_y f(x^{r},y^{r-1}) +\xi^{r} + \frac{1}{\rho}(y^{r}-y^{r-1})=0,
\end{align}
{where $\xi^r\in \partial(\mathcal{I}_{\mathcal Y}(y^r)+ g(y^r))$.}
We subtract these two  equalities, multiply both sides by $y^{r+1}-y^r$, utilize the defining property of subgradient vectors:
$
\langle \xi^{r+1}- \xi^{r}, y^{r+1} -y^r\rangle\ge 0,
$
and we obtain:
\begin{align*}
& \frac{1}{\rho}\langle v^{r+1},y^{r+1}-y^r\rangle\nonumber\\
& \le\langle\nabla_y f(x^{r+1},y^r)-\nabla_y f(x^r,y^r), y^{r+1}-y^r\rangle
\\
& \quad +\langle\nabla_y f(x^r,y^r)-\nabla_y f(x^r,y^{r-1}),y^{r+1}-y^r\rangle,
\end{align*}%
where $v^{r+1}$ is defined in \eqref{eq:v}.
{Applying  \eqref{eq.eqrel} to the LHS to the above expression, and using similar techniques as in   \eqref{eq.lip}, \eqref{eq.lip1} for the RHS of the above expression [note that this time we use a constant $\theta$ instead of $\rho$, when applying \eqref{eq.lip}], we obtain the following:
	{\small
\begin{align}\label{eq:rec}
&\frac{1}{2\rho}\|y^{r+1}-y^r\|^2 \le\hspace{-0.05cm}\frac{1}{2\rho}\hspace{-0.05cm}\|y^r-y^{r-1}\|^2\hspace{-0.1cm}-\hspace{-0.05cm} \frac{1}{2\rho}\|v^{r+1}\|^2\hspace{-0.1cm} +\hspace{-0.05cm}\frac{{L^2_y}}{2\theta}\|x^{r+1}-x^r\|^2\hspace{-0.1cm}\nonumber
\\
&+\frac{\theta}{2}\|y^{r+1}-y^{r}\|^2 +\frac{\rho L^2_y}{2}\|y^r-y^{r-1}\|^2+\frac{1}{2\rho}\|v^{r+1}\|^2-\theta\|y^r-y^{r-1}\|^2\nonumber
\\
&=\frac{1}{2\rho}\|y^{r}-y^{r-1}\|^2 \hspace{-0.1cm}+\frac{{L^2_y}}{2\theta}\|x^{r+1}-x^r\|^2\hspace{-0.1cm}+\frac{\theta}{2}\|y^{r+1}-y^r\|^2\hspace{-0.1cm} \nonumber  \\
&-\left(\theta-\frac{\rho L^2_y}{2}\right)\|y^r-y^{r-1}\|^2.
\end{align}}%

By combining Lemmas \ref{le.descent:x} and \ref{le.descent:y} we obtain
\begin{align*}
&\ell(x^{r+1},y^{r+1})-\ell(x^{r},y^r)\nonumber\\
& \le - \sum_{i=1}^{K}\left(\beta + \mu_i- \frac{L_{x_i}}{2} - \frac{\rho {L^2_y}}{2}\right)\|x_i^{r+1}-x_i^r\|^2\\
&+\frac{1}{\rho}\|y^{r+1}-y^r\|^2 -\left({\theta}-(\frac{1}{2\rho}+\frac{\rho L^2_y}{2})\right)\|y^r-y^{r-1}\|^2.
\end{align*}

Multiplying both sides of \eqref{eq:rec} by  $4/(\theta \rho)$, and adding the resulting inequality to the above expression, we have

\begin{align*}
&\ell(x^{r+1},y^{r+1}) + \frac{2}{\rho^2 \theta } \|y^{r+1}-y^{r}\|^2 \\
& \le \ell(x^{r},y^r) + \frac{2}{\rho^2 \theta } \|y^{r}-y^{r-1}\|^2 + \frac{3}{\rho}\|y^{r+1} - y^{r}\|^2\\
&  \quad - \sum_{i=1}^{K}\left(\beta + \mu_i- \frac{L_{x_i}}{2} -\left(\frac{2L^2_y}{\theta^2\rho}+\frac{\rho L^2_y}{2}\right) \right)\|x_i^{r+1}-x_i^r\|^2\\
& \quad -\left({\theta}-(\frac{1}{2\rho}+\frac{\rho L^2_y}{2})\right)\|y^r-y^{r-1}\|^2 \\
& \quad -\frac{4}{\theta\rho}\left(\theta-\frac{\rho L^2_y}{2}\right)\|y^r-y^{r-1}\|^2.
\end{align*}

Finally, adding in both sides the term $\left( {\frac{1}{2\rho}}-4\left(\frac{1}{\rho}-\frac{L^2_y}{2\theta}\right)\right)\|y^{r+1}-y^r\|^2$ and using  the definition of the potential function \eqref{potential_111}, we obtain the following

\begin{align*}
\notag
&\mathcal{P}^{r+1}
\le \mathcal{P}^{r}
+\left(\frac{3}{\rho}+ {\frac{1}{2\rho}}-4\left(\frac{1}{\rho}-\frac{L^2_y}{2\theta}\right)\right)\|y^{r+1}-y^r\|^2 \\
&- \sum_{i=1}^{K}\left(\beta + \mu_i- \frac{L_{x_i}}{2} -\left(\frac{2L^2_y}{\theta^2\rho}+\frac{\rho L^2_y}{2}\right) \right)\|x_i^{r+1}-x_i^r\|^2\nonumber.
\end{align*}

In the inequality above we do not include $\theta-\rho L^2_y/2$ (from RHS of the descent estimate in Lemma \ref{le.descent:x}) because by the choice of $\rho$ this term is positive.
Therefore, when
\begin{equation}
\rho<\frac{\theta}{4L^2_y},\quad \beta >   L^2_y\left(\frac{2}{\theta^2\rho}+ \frac{\rho}{2}\right)  + \frac{L_{x_{i}}}{2} - \mu_{i}, \; \forall i
\end{equation}
we have sufficient descent of the potential function $\mathcal{P}^{r+1}$.
This completes the proof. \QED

\subsection{Proof of Theorem \ref{th.main}}\label{app:T1}
We first bound the $i$th block of the optimality gap \eqref{eq.optgap} by
{\small
\begin{align*}
\notag
&\| ({\nabla} {\mathcal{G}^{\beta}_{\rho}}(x^r,y^r))_{i}\|
\\\notag
\le&  \beta\|x_{i}^{r+1}-x_{i}^r\|+\beta\|x_{i}^{r+1}-{\mbox{\rm Px}^{\beta}_{i}}(x_{i}^r-1/\beta\nabla_{x_i} f(x^r,y^r))\|
\\\notag
\mathop{\le}\limits^{(a)}&\beta\|x_{i}^{r+1}-x_{i}^r\|+\beta\|{\mbox{\rm Px}^{\beta}_{i}}({ x_i^{r}}-(\frac{1}{\beta}\nabla_{x_{i}} U_i(x_{i}^{r+1};w_{i}^{r+1},y^r))) \nonumber\\
&\quad
-{\mbox{\rm Px}^{\beta}_{i}}(x_i^r-\frac{1}{\beta}\nabla_{x_{i}} f(x^r,y^r))\|
\\\notag
\mathop{\le}\limits^{(b)}& \beta\|x_{i}^{r+1}-x_{i}^r\|+ {L_{u_i}}\|x_{i}^{r+1}-x_{i}^r\| + {L_{x_i}} \| w_{i}^{r+1} - x^{r} \| \\
\le & \left(\beta + {L_{u_i} + L_{x_i}} \right)\|x^{r+1}-x^r\|,
\end{align*}}%
where in $(a)$ we use the optimality conditions  w.r.t to $x_i$ in \eqref{eq:x-update}; in $(b)$
we use the nonexpansiveness of the  proximal operator,  $\nabla_{x_i} U_i(x_{i}^{r};w_{i}^{r+1},y^r) = \nabla_{x_i} f_i(w_i^{r+1}, y^r)$ (Assumption B2),  Assumption B4 (Lipschitz gradient), as well as the following identity
\begin{align*}
&\hspace{-3mm}\nabla_{x_{i}} U_i(x_{i}^{r+1};w_{i}^{r+1},y^r) -\nabla_{x_{i}} f(x^r,y^r)\nonumber\\
& =  \nabla_{x_{i}} U_i(x_{i}^{r+1};w_{i}^{r+1},y^r) - \nabla_{x_{i}} U_i(x_{i}^{r};w_{i}^{r+1},y^r) \nonumber\\
&\quad + \nabla_{x_{i}} U_i(x_{i}^{r};w_{i}^{r+1},y^r)-\nabla_{x_{i}} f(x^r,y^r).
\end{align*}}%
Moreover, utilizing the same argument for the optimality condition w.r.t to $y$ problem \eqref{eq:y-update}, we obtain:
\begin{align*}
\notag
&\| ({\nabla} {\mathcal{G}^{\beta}_{\rho}}(x^r,y^r))_{K+1}\|
\\\notag
%\le&{\frac{1}{\rho}}\|y^{r+1}-y^r\|+{\frac{1}{\rho}}\|y^{r+1}-{\mbox{\rm Py}^{1/\rho}}(y^r+ \rho \nabla_y f(x^r,y^r))\|
\notag
\mathop{\le}\limits^{(a)} & {\frac{1}{\rho}}\|y^{r+1}-y^r\|+{\frac{1}{\rho}}\| {\mbox{\rm Py}^{1/\rho}}({y^{r}}+\rho\nabla_y f(x^{r+1},y^{r})) \nonumber\\
&\quad -{\mbox{\rm Py}^{1/\rho}}(y^r+ \rho \nabla_y f(x^r,y^r))\|
\\\notag
\mathop{\le}\limits^{(b)}& {\frac{1}{\rho}}\|y^{r+1}-y^r\|+\|\nabla_y f(x^{r+1},y^{r})-\nabla_y f(x^r,y^r)\|
\\
\mathop{\le}\limits^{(c)}&  L_{y}  \|x^{r+1}-x^r\| + {\frac{1}{\rho}}\|y^{r+1}-y^r\|,
\end{align*}%
where in $(a)$ we use the optimality conditions  w.r.t $y$, in $(b)$ we use the nonexpansiveness of the  proximal operator and finally in (c) the Assumption A.3.
Combining \eqref{eq.descp} and the above two inequalities, we see that there exist constants $\sigma_1>0$ and $\sigma_2>0$ such that the following holds:
\begin{equation}
\|{{\nabla} {\mathcal{G}^{\beta}_{\rho}}}(x^r,y^r)\|^2\le \frac{\sigma_2}{\sigma_1}(\mathcal{P}^{r}-\mathcal{P}^{r+1}).
\end{equation}
Summing the above inequality over $r\in [T]$, we have
\begin{equation}
\sum^T_{{r=1}}\|{{\nabla} {\mathcal{G}^{\beta}_{\rho}}}(x^r,y^r)\|^2\le\frac{\sigma_2}{\sigma_1}({\mathcal{P}^1}-\mathcal{P}^{T+1})\le\frac{\sigma_2}{\sigma_1}(\mathcal{P}^1-{\underline{\ell}}),
\end{equation}
where in the last inequality we have used the fact that $\mathcal{P}^r$ is decreasing (by Lemma \ref{le.condalg1}) and lower bounded by  {$\underline{\ell}$}. The latter fact is because, when condition \eqref{eq:case1_cond} holds true, the coefficient in front of $\|y^{r+1}-y^r\|^2$ is positive, therefore $\mathcal{P}^{r+1}$ is lower bounded by $\ell(x^{r+1},y^{r+1})$, according to Assumption A1. By utilizing the definition $T(\epsilon)$, the above inequality becomes
 $
T(\epsilon)\epsilon^{2}\le\frac{\sigma_2}{\sigma_1}(\mathcal{P}^1-{\underline{\ell}}).
$

{Dividing both sides by $\epsilon^2$, the desired result is obtained.}\QED

\subsection{Proof of Lemma \ref{descentlemma3}}\label{app:lemma4}
Following similar steps as in Lemma \ref{le.descent:x} and using the assumption $\beta^{r} >  L_{x_i},   \forall i$ we obtain
\begin{align}\label{eq.descentx}
\ell(x^{r+1},y^{r})-\ell(x^{r},y^r)
\le - \left( \frac{\beta^{r}}{2} +\mu \right) \|x^{r+1}-x^r\|^2,
\end{align}
where $\mu := \min\limits_{i\in [K]} \mu_i$. To analyze the $y$ update, define
$$\ell'(x^{r+1},y)=f(x^{r+1},y)+ {\sum\limits_{i=1}^{K} h_i(x_i^{r+1})}-\mathcal{I}_{\cal{Y}}(y) -g(y).$$ The optimality condition for the $y$ update is
\begin{equation}\label{eq:opt:y:3}
\xi^{r+1}-\nabla_yf(x^{r+1},y^{r+1})+\frac{1}{\rho}(y^{r+1}-y^r)+\gamma^ry^{r+1}=0,
\end{equation}
where $\xi^{r+1}\in \partial (\mathcal{I}_{\cal{Y}}(y^{r+1})+ g(y^{r+1}))$.
Using this, we have the following series of inequalities: 
\begin{align}
\notag
&{\ell'(x^{r+1},y^{r+1})-\ell'(x^{r+1},y^r)}\nonumber\\
\stackrel{(a)}\le& \langle\nabla_y f(x^{r+1},y^r),y^{r+1}-y^r\rangle-\langle\xi^r,y^{r+1}-y^r\rangle\notag
\\
\notag
\mathop{=}\limits^{(b)}&\langle\nabla_y f(x^{r+1},y^r)-\nabla_y f(x^{r+1},y^{r+1}),y^{r+1}-y^r\rangle \nonumber\\
& \notag\quad +\frac{1}{\rho}\|y^{r+1}-y^r\|^2+\gamma^r\langle y^{r+1},y^{r+1}-y^r\rangle
\\\notag
&\quad +\langle\xi^{r+1}-\xi^r,y^{r+1}-y^r\rangle 
\\\notag
\mathop{=}\limits^{(c)}&\gamma^{r-1}\langle y^{r},y^{r+1}-y^r\rangle+ \frac{1}{\rho}\|y^{r+1}-y^r\|^2-\frac{1}{\rho}\langle v^{r+1},y^{r+1}-y^r\rangle
\\\notag
&+\langle\nabla_yf(x^{r+1},y^r)-\nabla_yf(x^r,y^r),y^{r+1}-y^r\rangle
\\\notag
\mathop{\le}\limits^{(d)}&\frac{1}{2\rho}\|y^r-y^{r-1}\|^2+\frac{\rho {L^2_y}}{2}\|x^{r+1}-x^r\|^2
\\
&\notag \quad -(\frac{\gamma^{r-1}}{2}-\frac{1}{\rho})\|y^{r+1}-y^r\|^2\\
&\quad + \frac{\gamma^r}{2}\|y^{r+1}\|^2-\frac{\gamma^{r-1}}{2}\|y^r\|^2+\frac{\gamma^{r-1}-\gamma^r}{2}\|y^{r+1}\|^2, \label{eq.descenty}
\end{align}}%
where $(a)$ uses the concavity of $\ell'(x,y)$;  in $(b)$ we use \eqref{eq:opt:y:3};
$(c)$  follows from \eqref{eq:opt:y:3}, the optimality condition for $y$ at iteration $r$ and plugging the resulting $\xi^{r+1}-\xi^r$; in $(d)$ we use the quadrilateral identity \eqref{eq.eqrel} for the term involving $v$,  and ignore the resulting negative term $-\frac{1}{2\rho}\|v^{r+1}\|^2$, and the Lipschitz continuity of $\nabla_y f$ (cf. {Assumption} A. 3), the Young's inequality, as well as the following identity: {\small
\begin{align*}
\notag
&\hspace{-0.2cm}\gamma^{r-1}\langle y^r,y^{r+1} \hspace{-0.1cm} -y^r\rangle=\frac{\gamma^{r-1}}{2}\left(\|y^{r+1}\|^2-\|y^r\|^2-\|y^{r+1}\hspace{-0.1cm}-y^r\|^2\right)
\\
&\hspace{-0.2cm}=\frac{\gamma^r}{2}\|y^{r+1}\|^2 \hspace{-0.1cm} -\frac{\gamma^{r-1}}{2}(\|y^r\|^2\hspace{-0.1cm}+\|y^{r+1}-y^r\|^2)+\left(\frac{\gamma^{r-1}\hspace{-0.2cm}-\gamma^r}{2}\right)\hspace{-0.1cm}\|y^{r+1}\|^2.
\end{align*}}%
Combining \eqref{eq.descentx} and  \eqref{eq.descenty}, we obtain the desired result. \QED

\subsection{Proof of Lemma \ref{potential3}}\label{app:lemma5}
To simplify notation, define $f^{r+1}:=f(x^{r+1},y^{r+1})$. The optimality conditions of $y$ problem are given by
\begin{subequations}
	\begin{align}
&\hspace{-0.3cm}\langle \nabla_y f^{r+1}\hspace{-0.2cm}-\frac{1}{\rho}(y^{r+1}\hspace{-0.2cm}-y^r)-\gamma^ry^{r+1}   -\vartheta^{r+1}, y^{r+1}\hspace{-0.2cm}-y\rangle\ge 0
\label{eq.yopt1} \\
&\hspace{-0.3cm} \langle \nabla_y f^r\hspace{-0.2cm}-\frac{1}{\rho}(y^{r}-y^{r-1})-\gamma^{r-1}y^{r}-\vartheta^{r},  y^r-y\rangle\ge 0,\label{eq.yopt2}
\end{align}
\end{subequations}
for all $y \in \mathcal{Y}$, where  $\vartheta^{r+1} \in \vartheta g(y^{r+1})$.

Plugging in $y=y^r$ in \eqref{eq.yopt1}, $y=y^{r+1}$ in \eqref{eq.yopt2}, adding them together and utilizing the defining property of subgradient vectors, i.e $\langle \vartheta^{r+1}- \vartheta^{r}, y^{r+1} -y^r\rangle\ge 0$, we obtain 
\begin{align}\label{eq.recur}
&\frac{1}{\rho}\langle v^{r+1}, y^{r+1}-y^r\rangle + \langle\gamma^ry^{r+1}-\gamma^{r-1}y^r, y^{r+1}-y^r\rangle\nonumber\\
&\le \langle\nabla_y f^{r+1}-\nabla_y f^r,y^{r+1}-y^r\rangle,
\end{align}

where $v^{r+1}$ is defined in \eqref{eq:v}.
In the following, we will use the above inequality to analyze the recurrence of the size of the difference between two consecutive iterates. First, we have {\small
\begin{align}
\notag
&\langle\gamma^ry^{r+1}-\gamma^{r-1}y^r,y^{r+1}-y^r\rangle\nonumber\\
=&\langle\gamma^ry^{r+1}-\gamma^ry^r+\gamma^ry^r-\gamma^{r-1}y^r,y^{r+1}-y^r\rangle\nonumber
\\\notag
=&\gamma^r\|y^{r+1}-y^r\|^2+(\gamma^r-\gamma^{r-1})\langle y^r,y^{r+1}-y^r\rangle
\\\notag
=&\gamma^r\|y^{r+1}-y^r\|^2+\frac{\gamma^r-\gamma^{r-1}}{2}\left(\|y^{r+1}\|^2-\|y^r\|^2-\|y^{r+1}-y^r\|^2\right)
\\
=&\frac{\gamma^{r}+\gamma^{r-1}}{2}\|y^{r+1}-y^r\|^2\hspace{-0.1cm}-\frac{\gamma^{r-1}-\gamma^r}{2}\hspace{-0.1cm}\left(\|y^{r+1}\|^2\hspace{-0.1cm}-\|y^r\|^2\right).\label{eq.gammar}
\end{align}}%
Substituting \eqref{eq.gammar} and \eqref{eq.eqrel} into \eqref{eq.recur}, we have
\begin{align*}\notag
&\frac{1}{2\rho}\|y^{r+1}-y^r\|^2-\frac{\gamma^{r-1}-\gamma^r}{2}\|y^{r+1}\|^2
\\
\le & \frac{1}{2\rho}\|y^r-y^{r-1}\|^2-\frac{1}{2\rho}\|v^{r+1}\|^2-\frac{\gamma^{r-1}-\gamma^r}{2}\|y^r\|^2 \nonumber\\
& \hspace{-0.4cm}-\frac{\gamma^{r-1}\hspace{-0.1cm}+\gamma^r}{2}\|y^{r+1}-y^r\|^2\hspace{-0.1cm} +\langle\nabla_y f^{r+1}-\nabla_y f^r,y^{r+1}-y^r\rangle\nonumber
\\
\mathop{\le}\limits^{(a)}&\frac{1}{2\rho}\|y^r-y^{r-1}\|^2-\gamma^r\|y^{r+1}-y^r\|^2-\frac{\gamma^{r-1}-\gamma^r}{2}\|y^r\|^2\nonumber\\
&\quad +\langle\nabla_yf(x^{r+1},y^r)-\nabla_yf(x^r,y^r),y^{r+1}-y^r\rangle
\\
\mathop{\le}\limits^{(b)}&\frac{1}{2\rho}\|y^r-y^{r-1}\|^2-\frac{\gamma^{r-1}-\gamma^r}{2}\|y^r\|^2\nonumber\\
& +\frac{{L^2_y}}{2\gamma^r}\|x^{r+1}-x^r\|^2-\frac{\gamma^r}{2}\|y^{r+1}-y^r\|^2
\end{align*}
where $(a)$ is true because of the fact that  $0< \gamma^r<\gamma^{r-1}$,  which implies that $-\frac{\gamma^{r-1}-\gamma^{r}}{2}<0$ and $-\frac{\gamma^{r-1}+\gamma^{r}}{2}<-\gamma^{r}$, and the concavity of function $f(x,y)$ in $y$; in $(b)$ we use the Young's inequality. 
Next, let us define 
$$\mathcal{F}^{r+1}:=\frac{1}{2\rho}\|y^{r+1}-y^r\|^2-\frac{\gamma^{r-1}-\gamma^r}{2}\|y^{r+1}\|^2.$$ 
Then we have{\small
\begin{align}\notag
&\hspace{-0.2cm}\frac{4\mathcal{F}^{r+1}}{\rho\gamma^r}\le \frac{2}{\rho^2\gamma^r}\|y^r-y^{r-1}\|^2-\frac{2}{\rho}\left(\frac{\gamma^{r-1}}{\gamma^r}-1\right)\|y^{r}\|^2\nonumber\\
&\quad -\frac{2}{\rho}\|y^{r+1}-y^r\|^2+\frac{2{L^2_y}}{\rho(\gamma^r)^2}\|x^{r+1}-x^r\|^2\nonumber
\\\notag
\hspace{-0.3cm}\le&\frac{4\mathcal{F}^{r}}{\rho\gamma^{r-1}}+\frac{2}{\rho^2}\left(\frac{1}{\gamma^r}-\frac{1}{\gamma^{r-1}}\right)\hspace{-0.1cm}\|y^r-y^{r-1}\|^2\hspace{-0.1cm}+\frac{2}{\rho}\hspace{-0.1cm}\left(\frac{\gamma^{r-2}}{\gamma^{r-1}}-\frac{\gamma^{r-1}}{\gamma^r}\right)\hspace{-0.1cm}\|y^r\|^2
\\
&-\frac{2}{\rho}\|y^{r+1}-y^r\|^2+\frac{2{L^2_y}}{\rho(\gamma^r)^2}\|x^{r+1}-x^r\|^2.\label{eq.diffiter}
\end{align}}
Furthermore, {adding} \eqref{eq.descentlemma2} and \eqref{eq.diffiter}, {and ignoring the negative term -$\frac{\gamma^{r-1}}{2}\|y^{r+1}-y^{r}\|^2$} we have{\small
\begin{align*}
\notag
&{{\ell}(x^{r+1},y^{r+1})}-\frac{\gamma^r}{2}\|y^{r+1}\|^2 +  \frac{4\mathcal{F}^{r+1}}{\rho\gamma^r}
\\\notag
\le & {{\ell}(x^r,y^r)}-\frac{\gamma^{r-1}}{2}\|y^r\|^2+  \frac{4\mathcal{F}^{r}}{\rho\gamma^{r-1}} -\frac{1}{\rho}\|y^{r+1}-y^r\|^2
\\\notag
&\hspace{-0.6cm}-\left({\frac{\beta^r}{2} + \mu}-\left(\frac{\rho {L^2_y}}{2}+\frac{2{L^2_y}}{\rho(\gamma^r)^2}\right)\right)\hspace{-0.1cm}\|x^{r+1}-x^r\|^2\hspace{-0.1cm} +\frac{\gamma^{r-1}-\gamma^r}{2}\|y^{r+1}\|^2
\\
&\hspace{-0.6cm}+\frac{2}{\rho}\left(\frac{1}{4}+\frac{1}{\rho\gamma^r}-\frac{1}{\rho\gamma^{r-1}}\right)\|y^{r}-y^{r-1}\|^2+\frac{2}{\rho}\left(\frac{\gamma^{r-2}}{\gamma^{r-1}}-\frac{\gamma^{r-1}}{\gamma^r}\right)\|y^r\|^2.
\end{align*}}%
Finally, by adding to both sides the term $\frac{2}{\rho}\left(\frac{1}{\rho\gamma^{r+1}} - \frac{1}{\rho\gamma^{r}} \right)\|y^{r+1}-y^{r}\|^2$, using the definition $\{\mathcal{P}^{r}\}$ in \eqref{eq:potential:3}, we obtain{\small
\begin{align*}
\notag
&\mathcal{P}^{r+1}- \mathcal{P}^{r}\le -\frac{1}{2\rho}\|y^{r+1}-y^r\|^2 +\frac{\gamma^{r-1}-\gamma^r}{2}\|y^{r+1}\|^2\nonumber\\
&-\left({\frac{\beta^r}{2} + \mu}-\left(\frac{\rho {L^2_y}}{2}+\frac{2{L^2_y}}{\rho(\gamma^r)^2}\right)\right)\|x^{r+1}-x^r\|^2\nonumber
\\
&+\frac{2}{\rho}\left(\frac{1}{\rho\gamma^{r+1}}-\frac{1}{\rho\gamma^{r}}\right)\|y^{r+1}-y^{r}\|^2+\frac{2}{\rho}\left(\frac{\gamma^{r-2}}{\gamma^{r-1}}-\frac{\gamma^{r-1}}{\gamma^r}\right)\|y^r\|^2.
\end{align*}}%
According to the above, to achieve descent in $\|y^{r+1}-y^r\|^2$ we need to ensure that the following holds:
\begin{equation}
-{1}/{2\rho}+{2}/{\rho^2}({1}/{\gamma^{r+1}}-{1}/{\gamma^r})<0\label{eq.shrink}.
\end{equation}
Note that, \eqref{eq.shrink} is equivalent to the condition
$
\frac{1}{\gamma^{r+1}}-\frac{1}{\gamma^r} \le {\rho}/{4}.
$
which holds by {condition \eqref{assumpt}}. 
This completes the proof.  \QED

\subsection{Proof of Theorem \ref{th.main3}}\label{app:T2}
{For simplicity, let $\mathcal{G}^r_{i}: = (\mathcal{G}^{\beta^{r}}_{\rho}(x^r,y^r))_{i}$. Similarly as in the proof of Theorem \ref{th.main}, we have
$$\|  \mathcal{G}^r_{i}\|
\mathop{\le} {\left(\beta^r + {L_{u_i} + L_{x_i}} \right)}\|x^{r+1}-x^r\|, \; \forall i\in[K].$$
For the corresponding bound for $y$ we have{\small
\begin{align}
\notag
&\|  {\nabla} \mathcal{G}^{r}_{K+1}\|
\\\notag
\le& {\frac{1}{\rho}}\|y^{r+1}-y^r\|+{\frac{1}{\rho}}\|y^{r+1}-{\mbox{\rm Py}^{1/\rho}}(y^r+ \rho \nabla_y f(x^r,y^r))\|
\\\notag
\mathop{=}\limits^{(a)} & {\frac{1}{\rho}}\|y^{r+1}-y^r\|+{\frac{1}{\rho}}\|{\mbox{\rm Py}^{1/\rho}}(y^{r}+\rho\nabla_y f(x^{r+1},y^{r+1})-\rho \gamma^r y^{r+1})\nonumber\\
&-{\mbox{\rm Py}^{1/\rho}}(y^r+ \rho \nabla_y f(x^r,y^r))\|\nonumber
\\\notag
\mathop{\le}\limits^{(b)}&  L_{y} \|x^{r+1}-x^r\| + \left({\frac{1}{\rho}}+L_{y} \right)\|y^{r+1}-y^r\| + \gamma^r\|y^{r+1}\|,
\end{align}}%
where in $(a)$ we use the optimality conditions  w.r.t $y$; in $(b)$ we use the nonexpansiveness of the  proximal operator, as well as the the Lipschitz gradient condition w.r.t $y$ two times. Combining the above two bounds we obtain{\small
\begin{align}
&\| {\nabla} \mathcal{G}^r\|^{2}  \leq
\sum\limits_{i=1}^{K} {\left(\beta^r + {L_{u_i} + L_{x_i}} \right)^2} \|x^{r+1}-x^r\|^{2} \nonumber \\
&+3(\gamma^r)^{2}\|y^{r+1}\|^{2} + 3L_{y}^{2} \|x^{r+1}-x^r\|^{2} + 3\left({\frac{1}{\rho}}+ L_{y} \right)^{2}\hspace{-0.2cm}\|y^{r+1}-y^r\|^{2}   \nonumber\\
& \leq \left(K({L +\beta^{r}})^{2}+ 3L_{y}^{2}\right) \|x^{r+1}-x^r\|^{2} \nonumber\\
&+ 3\left({\frac{1}{\rho}}+ L_{y} \right)^{2}\|y^{r+1}-y^r\|^{2} +3(\gamma^r)^{2}\|y^{r+1}\|^{2}, \label{eq:size:G}
\end{align}}
\hspace{-0.1cm}where we defined {$L := \hspace{-0.1cm} \max\limits_{i\in [K]} (L_{u_i} + L_{x_i})$.}}
Moreover, we choose
\begin{align}\label{eq:beta:r}
\beta^r=\rho {L^2_y}+\frac{{2 \kappa L^2_y}}{\rho(\gamma^r)^2}  -2\mu,
\end{align}
where $\kappa$ is chosen to satisfy
$\kappa >2,  \quad \beta^{0} >  L_{x_i},  \; \forall i$.

{By condition \eqref{eq.conditions31}, it is clear that $\beta^{r+1} \geq \beta^{r}$. Combining this with the choice of $\kappa$ we have: $\beta^{r} \geq \beta^{0}>L_{x_i}, \; \forall i,r$.}  Thus, this choice of $\beta^{r}$ satisfies Assumption C-3.

Moreover, such a choice implies that
\begin{align}\label{eq:alpha}
\alpha^r:=\frac{\beta^r}{2} +  \mu-\bigg(\frac{\rho {L^2_y}}{2}+\frac{2{L^2_y}}{\rho(\gamma^r)^2}\bigg)=\frac{{(\kappa-2) L^2_y}}{\rho(\gamma^r)^2}.
\end{align}
Using these properties in \eqref{eq:size:G}, the constants in front of $\|x^{r+1}-x^r\|^2$ becomes
{{
{\begin{align}
\notag
&K( L +\beta^{r})^{2}+ 3L_{y}^{2}= K \left(L + \rho  L^2_y+\frac{ 2 \kappa L^2_y}{\rho(\gamma^r)^2}  -2\mu\right)^{2} + 3 L_{y}^{2}
\\\notag
& \stackrel{(a)}=  \left(K^{2}L+ \rho K^2 L_y -2\mu K^{2} + K^2\frac{2\kappa }{\kappa-2}\alpha^{r} \right)^{2} + 3 L_{y}^{2}
\\
&\stackrel{(b)}\le (d_1 \alpha^r)^{2}\label{eq.bdofk}
\end{align}}%
\noindent in  $(a)$ we use the identity shown in \eqref{eq:alpha};  $(b)$ always holds for some $d_1>1$ (which are both independent of $r$), since $\alpha^r$ is an increasing sequence, and $\alpha^0$ is bounded away from zero.}  Note that {since} $y$ lies in a bounded set, there exists $\sigma_y$ such that $\|y^{r+1}\|^2\le \sigma^2_y, \forall~ r$. Using \eqref{eq.bdofk}, setting $z:= 3\left(L_{y}+ \frac{1}{ \rho} \right)^{2}$, we obtain
\begin{align}
\| {\nabla} \mathcal{G}^r\|^{2}
&\le (d_1 \alpha^r)^{2}  \|x^{r+1}-x^r\|^{2}\nonumber\\
& \quad \quad + z\|y^{r+1}-y^r\|^{2} + 3(\gamma^r)^{2}\sigma_{y}^{2}  \label{eq:bound:G}
\end{align}}
Furthermore, when $\beta^r=\rho {L^2_y}+\frac{{2 \kappa L^2_y}}{\rho(\gamma^r)^2} -2\mu$ and since $\frac{1}{\gamma^{r+1}}-\frac{1}{\gamma^r}\le {\frac{\rho}{5}}$, the bound of the potential function  \eqref{eq:whole:2} becomes
\begin{align}
\mathcal{P}^{r+1} &\le \mathcal{P}^{r}-\frac{1}{10\rho}\|y^{r+1}-y^r\|^2-\alpha^r\|x^{r+1}-x^r\|^2\nonumber\\
&+\frac{\gamma^{r-1}-\gamma^r}{2}\|y^{r+1}\|^2
+\frac{2}{\rho}\left(\frac{\gamma^{r-2}}{\gamma^{r-1}}-\frac{\gamma^{r-1}}{\gamma^r}\right)\|y^r\|^2.\nonumber
\end{align}

{Because $\{\alpha^r\}$ is increasing and  $\|y^r\|^2\le\sigma^2_y$, the above relation implies the following
\begin{align}
&\frac{1}{10\rho}\|y^{r+1}-y^r\|^2+\alpha^r\|x^{r+1}-x^r\|^2\label{eq.recc}\\
&\le\mathcal{P}^r-\mathcal{P}^{r+1}+\frac{\gamma^{r-1}-\gamma^r}{2}\sigma^2_y
+\frac{2}{\rho}\left(\frac{\gamma^{r-2}}{\gamma^{r-1}}-\frac{\gamma^{r-1}}{\gamma^r}\right)\sigma^2_y.\nonumber
\end{align}
Let us define
\begin{align*}
{d_2^r} := \min\left\{\frac{1}{10\rho}, 1\right\}/\max\left\{z, d^2_1 \alpha^r \right\}.
\end{align*}
Then by combining \eqref{eq.recc} and \eqref{eq:bound:G}, we obtain
\begin{align}
&\| {\nabla} \mathcal{G}^r\|^{2}\times {d_2^r} \le  \mathcal{P}^r-\mathcal{P}^{r+1}\nonumber\\
&+\frac{\gamma^{r-1}-\gamma^r}{2}\sigma^2_y
+\frac{2}{\rho}\left(\frac{\gamma^{r-2}}{\gamma^{r-1}}-\frac{\gamma^{r-1}}{\gamma^r}\right)\sigma^2_y + 3(\gamma^r)^2\sigma^2_y  \times {d_2^r}\nonumber
\end{align}}
Summing both sides from  $r=1$ to $T$, and noting that condition \eqref{eq.conditions31} implies  $\frac{\gamma^{r}}{\gamma^{r+1}}\le 1.2, \; \forall r$, we obtain 
\begin{align*}
& \sum^T_{{r=1}}{d_2^r}\|{{\nabla} \mathcal{G}^{r}}\|^2 \le \sum^{T}_{{r=1}}{d_2^r \frac{3(\kappa-2) L_y^2 \sigma_{y}^{2}}{ \rho \alpha^{r}}} + \nonumber \\
& + \mathcal{P}^1-\underbar{$\mathcal{P}$}+\sigma_y^{2} \left(\frac{{\gamma^{0}}-\gamma^T}{2}+\frac{2}{\rho}\left(\frac{{\gamma^{-1}}}{{\gamma^{0}}}-\frac{\gamma^{T-1}}{\gamma^T}\right)\right) \nonumber,
\end{align*}
{where we have defined 
	$$d_3 := \mathcal{P}^1-\underline{\mathcal{P}}+\sigma_y^{2} \left(\frac{{\gamma^{0}}-\gamma^T}{2}+\frac{2}{\rho}\left(\frac{{\gamma^{-1}}}{{\gamma^{0}}}-\frac{\gamma^{T-1}}{\gamma^T}\right)\right);$$ $\underbar{$\mathcal{P}$}$ is a lower bound of $\mathcal{P}$, which is a finite number due to the lower boundness assumption of $\underbar{$\ell$}$ and the compactness of $\mathcal{Y}$ (see Assumption A.1).}
Notice that since $d_1>1$, we have  
$$d_2^r \leq \frac{d_4}{d^2_1 a^r} \leq \frac{d_4}{a^r},$$ 
where  $d_4 := \min\left\{\frac{1}{10\rho}, 1\right\}$. Also, there exists $d_5 > \max\left\{\frac{d_1^2}{d_4},\frac{z}{d_4a^{0}} \right\}$ such that $d_2^r \geq \frac{1}{d_5 \alpha^r}$.

By utilizing the definition of $T(\epsilon)$ {and the above bounds}, we know that
{
\begin{equation} \label{eq:rate_case3}
\epsilon^{2} \le\frac{d_3d_5+ \frac{3 d_4 d_5(\kappa-2) L_y^2 \sigma_{y}^{2}}{ \rho }\sum^{T(\epsilon)}_{ r=1}\frac{1}{(\alpha^r)^2}}{\sum^{T(\epsilon)}_{{r=1}}\frac{1}{\alpha^r}}.
\end{equation}
}
Moreover, when $\gamma^r = \frac{1}{\rho r^{1/4}}$, it can be verified that the following holds:
$$\frac{1}{\gamma^{r+1}}-\frac{1}{\gamma^r}\le 0.19 \rho,\quad \forall r\ge 1,$$
because $(r+1)^{1/4}-(r)^{1/4}$ is a monotonically decreasing function and its maximum value is achieved at $r=1$. We can plug in this choice of $\gamma^r$ into \eqref{eq:alpha},  and obtain  
$$\alpha^r={(\kappa-2)\rho L^2_y\sqrt{r}}.$$ 

Using these choices of $\{\gamma^r, \alpha^r\}$, and by utilizing the bounds that $\sum_{r=1}^{T} 1/r \leq c\ln(T)$ (for some $c>0$), and $\sum_{r=1}^{T} 1/\sqrt{r} \geq \sqrt{T} $,
{the relation \eqref{eq:rate_case3}} becomes:

\begin{equation}\label{eq:bound123}
\epsilon^{2} \le\frac{C\log(T(\epsilon))}{\sqrt{T(\epsilon)}},
\end{equation}
where $C>0$ is some constant independent of the iteration. {Then, the desired result follows directly from \eqref{eq:bound123}}. \QED

\end{document}

\maketitle

\section{Preliminaries}
In this document, we consider the problem 
\begin{align}\label{eq:problem}
\begin{split}
\min_{{x}}\max_{{y}}&\quad f(x_1,x_2,\cdots, x_K, y) + \sum_{i=1}^{K}h_i(x_i) - g(y)\\
\mbox{s.t.} & \quad x_i\in \mathcal{X}_i, \; y\in \mathcal{Y}, \; i=1,\cdots, K
\end{split},
\end{align}
where $f:\mathbb{R}^{{N K}+M}\to \mathbb{R}$ is a continuously differentiable function; $h_i: \mathbb{R}^{N}\to \mathbb{R}$ and $g:\mathbb{R}^{M}\to \mathbb{R}$ are some convex possibly non-smooth functions; $x:=[x_1;\cdots; x_K]\in\mathbb{R}^{{N\cdot K}}$ and $y\in\mathbb{R}^M$ are the block optimization variables; $\mathcal{X}_i$'s and $\mathcal{Y}$ are some convex {and compact} feasible sets. 
For notational simplicity, we will use $\ell(x_1,x_2,\cdots, x_K, y)$ to denote the overall objective function for problem \eqref{eq:problem}.

In this document, we analyze a version of the HiBSA algorithm proposed in \cite{lu2019hybrid_ieee}, in which the following problem is solved {approximately}  when updating $y$:
\begin{align}\label{eq:max}
{\widetilde{y}^{r+1}} = \arg\max_{y \in \mathcal{Y}}\{f({x^{r+1}},y)-g(y) - \frac{\gamma^{r}}{2} \|y\|^{2}\}.
\end{align}
In particular, a finite number of gradient ascent steps will be performed so that problem \eqref{eq:max} is solved to certain accuracy. Compared with the version proposed in the main paper \cite{lu2019hybrid_ieee}, which solves the above problem exactly at each iteration, some major benefits of such an approach are given below: 
\begin{itemize}
	\item One does not need to know the precise form of $f(x,y)$, but only need to have access to the gradients $\nabla_x f(x,y)$ and $\nabla_y f(x,y)$;
	\item Since only gradient evaluations are needed, it is more convenient to compare the resulting complexity bounds with those given in the existing works; see \cite[Table I]{lu2019hybrid_ieee}.
\end{itemize}

More precisely, the modified HiBSA algorithm to be analyzed in this document takes the following form:
\begin{center}
\fbox{\small
\vspace{-1cm}
\begin{minipage}{0.95\linewidth}
\textsf{\bf Hybrid Block Successive Approximation (HiBSA) Algorithm}\\ 
\textsf{(multiple ascent steps for the maximization problem)} \\
At each iteration $r=1, 2, 3,\cdots$ \\
\; {\bf [S1].} For $i=1,\cdots, K$, perform the following update:
\begin{align}\label{eq:x-update}
\hspace{-0.3cm}	x_i^{r+1}\hspace{-0.1cm}	 =&\arg\hspace{-0.1cm}	\min_{x_i\in\mathcal{X}_i}\hspace{-0.1cm}	 U_{i}(x_{i};  w^{r+1}_i, y^r) + h_i(x_i) \hspace{-0.05cm}	 +\hspace{-0.05cm}	  \frac{\beta^r}{2}\|x_i-x_i^r\|^2
\end{align}		
					
\; {\bf [S2].} Perform ${J^{r}}$ updates for the $y$-block, i.e for $j=1,\ldots,{J^{r}}$ ($y^{r+1,1} = y^{r}$) :
{{\small
\begin{align}\label{eq:y-update}
y^{r+1,j+1}=\arg\max_{y\in\mathcal{Y}}\;\Big\{  &\langle \nabla_y f(x^{r+1}, y^{r+1,j}) - \gamma^{r}y^{r+1,j}, y - y^{r+1,j} \rangle   \nonumber\\ \vspace{-2mm}
&- g(y)-\frac{1}{2\rho}\|y-y^{r+1,j}\|^2\Big\}
\end{align}}}
\hspace{1mm}$y^{r+1} = y^{r+1,{J^{r}}}$	

\; {\bf [S3].} If  converges, stop; otherwise, set  $r=r+1$, go to  {\bf [S1]}.
				\end{minipage}
			}
		\end{center}

Below we provide the main assumptions. 

\noindent{\underline{\sf{Assumption A}}.} The following conditions hold for \eqref{eq:problem}:
\begin{itemize}
	\item[A.1] $f:\mathbb{R}^{KN+M}\to \mathbb{R}$ is continuously differentiable; The feasible sets $\mathcal{X}  = \mathcal{X}_{1} \times \ldots \times \mathcal{X}_{K}$ and $\mathcal{Y} \subseteq \mathbb{R}^{M}$ are convex and compact. {Further $\ell(x,y)$ is lower bounded, that is, $\ell(x,y)\ge \underline{\ell},\; \forall~x\in \mathcal{X}, y\in \mathcal{Y}$.}
	
	\item[A.2] $h_i(\cdot)$'s and $g(\cdot)$  are convex and non-smooth functions;
	
	\item[A.3]$f$ has Lipschitz continuous gradient with respect to (w.r.t.) $x_{i}$ for every $i$ with constant $L_{x_{i}}$, that is:
	\vspace{-0.10cm}
	\begin{align}
	\hspace{-0.9cm}\| \nabla_{x_{i}} f(\bar{z}) - \nabla_{x_{i}}f(z) \| \leq L_{x_{i}} \| \bar{z} - z \|, \forall \bar{z}, z \in  \mathcal{X} \times \mathcal{Y};
	\end{align}
	Furthermore, $f$ has Lipschitz continuous gradient w.r.t. $y$ with constant $L_{y}$, that is:
	\vspace{-0.10cm}
	\begin{subequations}
		\begin{align}
		\hspace{-0.9cm}{\| \nabla_{y} f(\bar{z}) - \nabla_{y}f(z) \| \leq L_{y} \| \bar{z} - z \|, \forall \bar{z}, z \in  \mathcal{X} \times \mathcal{Y}} \label{eq:Ly}.
		\end{align}	
	\end{subequations}
	\vspace{-0.6cm}
\end{itemize}

\noindent{\underline{\sf{Assumption B}}.} Each $U_i(\cdot)$ satisfies the following conditions:
\begin{itemize}
	\item[B.1] {\bf (Strong convexity).}  Each $U_{i}(\cdot; w, y)$ is strongly convex with modulus $\mu_{i}>0$: {\small
		\begin{align*}
		& U_{i}(x_i; w, y) - U_{i}(z_i; w, y) \ge \langle \nabla_{z_i} U_{i}(z_i; w, y), x_i-z_i \rangle \nonumber\\
		&\quad + \frac{\mu_i}{2}\|x_i-z_i\|^2, \; \; \forall~w \in \mathcal{X}, y\in \mathcal{Y}, x_i, z_i\in \mathcal{X}_{i}.
		\end{align*}}
	\item[B.2] {\bf (Gradient consistency).} Each $U_{i}(\cdot; w,y)$ satisfies:
	\begin{align*}
	\hspace{-0.8cm}{\nabla_{z_i} U_{i}({z_i}; x, y)\Bigr|_{z_i=x_i}} = \nabla_{x_i} f(x,y), \; \forall~i, \; \forall~x\in \mathcal{X}, \; y\in \mathcal{Y}.
	\end{align*}
	\item[B.3] {\bf (Tight upper bound).} Each $U_{i}(\cdot; w,y)$ satisfies:
	\begin{align*}
	U_{i}(z_i; x, y) & \ge f(x,y), \;  \mbox{and}\; U_{i}(x_i; x, y) = f(x,y), \nonumber\\
	&\quad \; \forall~ x\in \mathcal{X}, y\in \mathcal{Y}, \; z_i\in \mathcal{X}_{i}.
	\end{align*}
	\item[B.4] {\bf (Lipschitz gradient).} Each $U_{i}(\cdot; w,y)$ satisfies:
	\begin{align*}
	&\|\nabla U_{i}(z_i; w, y) - \nabla U_{i}(v_i; w, y)\|\le L_{u_i}\|v_i-z_i\|, \\
	&\quad \; \forall~ w\in \mathcal{X}, y\in \mathcal{Y}, v_i, z_i\in \mathcal{X}_{i}.
	\end{align*}
\end{itemize}

\noindent\underline{\sf{Assumption C.}} Suppose that the following conditions hold:
\begin{align}
\gamma^{r} & \ge\gamma^{r+1}, \; {\gamma^{r}}\rightarrow 0, \; \sum^{\infty}_{r=1} \gamma^r=\infty\\
\beta^{r}  & >  \frac{c}{\gamma^{r-1}} + d,
\end{align}

Recall that the $\epsilon$ stationarity condition we are interested in is given below:
\begin{align}\label{eq:opt}
\|{\nabla} {\mathcal{G}^{\beta}_{\rho}}(x,y)\|\le \epsilon
\end{align}
where $\epsilon>0$ is a small constant, and we have defined the following gap function:
\begin{equation}\label{eq.optgap}
{\nabla} {\mathcal{G}^{\beta}_{\rho}}(x,y):=\left[
\begin{array}{c}\beta(x_1-\mbox{\rm Px}^{\beta}_{1}(x_1-1/\beta\nabla_{x_1}f(x,y)))
\\
\vdots
\\
\beta(x_K-\mbox{\rm Px}^{\beta}_{K}(x_K-1/\beta\nabla_{x_K} f(x,y) ))
\\
1/\rho(y - \mbox{\rm Py}^{1/\rho}(y+ \rho\nabla_y f(x,y)))\end{array}\right].
\end{equation}
where 
\begin{align}\label{eq:prox}
&\mbox{\rm Px}^{\beta}_{i}(v_{i}) := \arg\min_{x_{i} \in \mathcal{X}}\;
h_i(x_i) +\frac{\beta}{2}\|x_{i}-v_i\|^2, \; \forall~i\in[K]\\
&\mbox{\rm Py}^{1/\rho}(w) := \arg\max_{y\in \mathcal{Y}}\;  - g(y) - \frac{1}{2\rho}\|y-w\|^2. \nonumber
\end{align}%

\section{Convergence Analysis}
{
\begin{lemma}\label{lemma:constants}
Consider the following functions:
\begin{align*}
&f_{\gamma^{r}}(x,y) := f(x,y)-\frac{\gamma^r}{2}\|y\|^2 \\
&r_{\gamma^{r}}(x,y) : =f_{\gamma^{r}}(x,y)-g(y).
\end{align*}
It holds that $f_{\gamma^{r}}(x,y)$ and $r_{\gamma^{r}}(x,y)$ are strongly concave  in $y$ with strong convexity constant being $\gamma^{r}$, while $f_{\gamma^{r}}(x,y)$ has Lipschitz {continuous} gradient in $y$ with constant $L_y + \gamma^{r}$.
\end{lemma}

\begin{proof}
The proof is straightforward. \hfill $\blacksquare$
\end{proof}
}
{
\begin{lemma}\label{lemma:phi_prop}
{Let us define 
\begin{align*}
&\sigma_{\gamma^r}(x): = \max\limits_{y \in \mathcal{Y}} r_{\gamma^{r}}(x,y) \\
&\phi_{\gamma^r}(x): = \sigma_{\gamma^r}(x)  + \sum_{i=1}^{K} h_{i}(x_i).
\end{align*}
}
The function $\sigma_{\gamma^r}(x)$ is differentiable, that is
\begin{align}\label{eq:sigma:gradient}
\nabla_{x_{i}} \sigma_{\gamma^{r}}(x) := \nabla_{x_{i}} r_{\gamma^{r}}(x,\widetilde{y}),\; \forall~i
\end{align}
where
\begin{align}
\widetilde{y} := \arg\max_{y \in \mathcal{Y}}\{f(x,y)-g(y) - \frac{\gamma^{r}}{2} \|y\|^{2}\}.
\end{align}
Further, $\sigma_{\gamma^r}(x)$ has Lipschitz {continuous} gradient in $x_i$ with constant $\widetilde{L}_{x_i}^{r} = 2L_{x_i} +\frac{L_{x_i} L_y}{\gamma^{r} }$. Finally, it holds that { $$\phi_{\gamma^{r}}(x) \leq \phi_{\gamma^{r-1}}(x)  + \frac{1}{2}(\gamma^{r-1}-\gamma^{r}) d^2,$$}
for some constant $d$. 
\end{lemma}

\begin{proof}
First of all, the differentiability of $\sigma_{\gamma^r}(x)$ follows from the fact that the problem $\arg\max\limits_{y \in \mathcal{Y}} r_{\gamma^{r}}(x,y)$ has a unique solution, and Danskin's theorem \cite[Theorem A.1]{madry2017towards}. 
{Secondly, the fact that it has Lipschitz {continuous} gradient in $x_i$ is also well-known; see, e.g.,  \cite[Lemma B.1]{nouiehed2019solving}.}

Finally, note that
%\vspace{-0.2cm}
\begin{align}\label{sigma_ineq}
&\sigma_{\gamma^{r}}(x)\nonumber\\
& = \max\limits_{y \in \mathcal{Y}} \{ f(x,y) -g(y)-\frac{\gamma^{r-1}}{2}\|y\|^2+\frac{\gamma^{r-1}-\gamma^{r}}{2}\|y\|^2\} \\
&\leq \sigma_{\gamma^{r-1}}(x)  + \max\limits_{y \in \mathcal{Y}} \{\frac{\gamma^{r-1}-\gamma^{r}}{2}\|y\|^2\} \\
& \leq \sigma_{\gamma^{r-1}}(x) + \frac{\gamma^{r-1}-\gamma^{r}}{2} d^2,
\end{align}%
%\vspace{-1mm}
since $\gamma^{r-1}\geq \gamma^{r}$,  and $d:=\max\{\|y\|\mid y\in \mathcal{Y}\}$.  
{By adding $\sum_{i=1}^{K} h_{i}(x_i)$  to both sides of \eqref{sigma_ineq} we complete our proof.}
\hfill $\blacksquare$
\end{proof}
}

\vspace{1mm}
{
Next we proceed to the main results.
\begin{proposition}[Analysis of the minimization problem]\label{prop.main3}
{\it Suppose that {Assumptions A -- C hold}.  Let  $\{y^{r}\}$ be the iterate generated by HiBSA. Assume that we can obtain a {$\delta^{r} = \frac{1}{r^{1/3}}$} accurate solution of the  maximization problem \eqref{eq:max} at iteration $r$.  That is, suppose that the following hold:
\begin{align}\label{eq:y}
\|y^{r} - \widetilde{y}^{r} \| \leq \delta^{r}, \;  {\rm where} \; \widetilde{y}^{r} : = \arg \max\limits_{y \in \mathcal{Y}} r_{\gamma^{r-1}}(x^{r},y).
\end{align}
For a given $\epsilon>0$, let
$$T(\epsilon) \:= \min\{r \geq 1 \mid \|{{\nabla} \mathcal{G}^{\beta^{r}}}({x^{r+1},y^{r+1}})\|\le\epsilon\}.$$
Then, {$T(\epsilon) = \widetilde{\mathcal{O}}\left(\frac{1}{\epsilon^3} \right).$}
}
\end{proposition}
}
\begin{proof}
{
{First, from our assumptions $h$ is a convex function, so the following holds:
\begin{align}\label{eq:hsubgr}
h_{i}(x_{i}^{r}) \geq h_i(x_{i}^{r+1}) + \langle \vartheta^{r+1}_i, x_{i}^{r} - x_{i}^{r+1}\rangle,
\end{align}
where $\vartheta^{r+1}_i\in \partial h_i(x^{r+1}_i)$.}
{Also, from Lemma \ref{lemma:phi_prop} we know that {$\sigma_{\gamma^{r-1}}(x_{i},x_{-i})$} has Lipschitz {continuous} gradient w.r.t ${x_i}$ with constant {$\widetilde{L}_{x_i}^{r-1}$}.}
{As a result, combining the descent lemma for $\sigma_{\gamma^{r-1}}(x)$ with \eqref{eq:hsubgr} we obtain}
%\vspace{2mm}
\begin{align}\label{eq:prop33}
&\phi_{\gamma^{r-1}}(x_{i}^{r+1},w_{-i}^{r+1}) - \phi_{\gamma^{r-1}}(x_{i}^{r},w_{-i}^{r+1})   \\
&\leq \langle \nabla_{x_{i}} \sigma_{\gamma^{r-1}}(x_{i}^{r},w_{-i}^{r+1}) + \vartheta^{r+1}_i, x_{i}^{r+1} - x_{i}^{r} \rangle + \frac{\widetilde{L}_{x_{i}}^{r-1}}{2} \| x_{i}^{r+1} - x_{i}^{r} \|^{2} \nonumber\\
& = \langle \nabla_{x_{i}} f(w_{i}^{r+1}, {\hat{y}_i^{r}}) + \vartheta^{r+1}_i, x_{i}^{r+1} - x_{i}^{r} \rangle + \frac{\widetilde{L}_{x_{i}}^{r-1}}{2} \| x_{i}^{r+1} - x_{i}^{r} \|^{2}\nonumber
\end{align}%
{where we have defined}
\begin{align} \label{def_yhat_i}
&{\hat{y}^{r}_i}: = \arg \max\limits_{y \in \mathcal{Y}}\; r_{\gamma_{r-1}}(x_i^r,w_{-i}^{r+1},y),\; \; w_{i}^{r+1} : = (x_{i}^{r},w_{-i}^{r+1}),\nonumber\\
&\nabla_{x_{i}} \sigma_{\gamma^{r-1}}(x_{i}^{r},w_{-i}^{r+1}) := \nabla_{x_{i}} r_{\gamma^{r-1}}(x_{i}^{r},w_{-i}^{r+1},{\hat{y}^{r}_i}).
\end{align}
{Note that \eqref{def_yhat_i} follows from Lemma \ref{lemma:phi_prop} for $x = w_{i}^{r+1}$.} %

Second, from the optimality condition for the update step of $x_i$  in \eqref{eq:x-update}, we obtain
\begin{align}\label{eq:eq3}
&\langle \nabla_{x_{i}} U_{i}(x_{i}^{r+1}; w_{i}^{r+1}, y^{r})  + \vartheta_i^{r+1} \nonumber\\
& \quad\quad \quad  + \beta^{r}(x_{i}^{r+1} - x_{i}^{r}), x_{i}^{r} - x_{i}^{r+1} \rangle \geq 0.
\end{align}%
Thus by adding and subtracting  $\langle \nabla_{x_{i}} U_{i}(x_{i}^{r}; w_{i}^{r+1}, y^{r}), x_{i}^{r} - x_{i}^{r+1}\rangle $ in \eqref{eq:eq3} and by applying Assumptions B.1, B.2, we obtain the following:
\begin{align*}
&\langle \nabla_{x_{i}} f(w_{i}^{r+1}, y^{r}), x_{i}^{r+1} - x_{i}^{r} \rangle + \langle \vartheta^{r+1}_i, x_{i}^{r+1} - x_{i}^{r} \rangle \nonumber\\
&\leq
-\mu \| x_{i}^{r+1} - x_{i}^{r} \|^{2} - \beta^{r}  \|x_{i}^{r+1} - x_{i}^{r}\|^{2},
\end{align*}
with $\mu := \min\limits_{i\in [K]} \mu_i$.
Moreover, adding and subtracting  $\langle \nabla_{x_{i}} f(w_{i}^{r+1}, {\hat{y}^{r}_i}), x_{i}^{r+1} - x_{i}^{r}\rangle $  {to the above expression leads to } 
\begin{align}\label{eq:11111}
&\langle \nabla_{x_{i}} f(w_{i}^{r+1}, {\hat{y}^{r}_i}) + \vartheta^{r+1}_i, x_{i}^{r+1} - x_{i}^{r} \rangle \nonumber\\
&\leq
-\mu \| x_{i}^{r+1} - x_{i}^{r} \|^{2} - \beta^{r}  \|x_{i}^{r+1} - x_{i}^{r}\|^{2} \nonumber \\
& + \langle \nabla_{x_{i}} f(w_{i}^{r+1}, {\hat{y}^{r}_i}) - \nabla_{x_{i}} f(w_{i}^{r+1}, y^{r}), x_{i}^{r+1} - x_{i}^{r} \rangle \nonumber\\
&\leq -(\mu + \beta^{r}) \| x_{i}^{r+1} - x_{i}^{r} \|^{2} + \frac{1}{2\gamma^{r-1}}\| x_{i}^{r+1} - x_{i}^{r} \|^{2} \nonumber\\
&+ \frac{\gamma^{r-1} L_{x_i}^{2}}{2} \|{\hat{y}^{r}_i} - y^{r} \|^{2},
\end{align}
{where the last inequality follows from Young's inequality and the Lipschtiz gradient property of $f(\cdot)$.}
Then, combining \eqref{eq:11111} with \eqref{eq:prop33} results in
{\small
\begin{align*}
&\phi_{\gamma^{r-1}}(x_{i}^{r+1},w_{-i}^{r+1}) - \phi_{\gamma^{r-1}}(x_{i}^{r},w_{-i}^{r+1}) \leq  \frac{\gamma^{r-1} L_{x_i}^{2}}{2} \| {\hat{y}^{r}_i} - y^{r} \|^{2} \nonumber \\
&-\left(\beta^r +\mu - \frac{1}{2\gamma^{r-1}} - \frac{\widetilde{L}_{x_{i}}^{r-1}}{2}\right) \| x_{i}^{r+1} - x_{i}^{r} \|^{2}.
\end{align*}}%
Summing over $i\in [K]$ leads to
{\small
\begin{align}\label{eq:eq4}
&\phi_{\gamma^{r-1}}(x^{r+1}) - \phi_{\gamma^{r-1}}(x^{r}) \leq \frac{\gamma^{r-1} \widetilde{L}_{x}^{2}}{2}\sum\limits_{i=1}^{K} \| {\hat{y}^{r}_i} - y^{r} \|^{2} \nonumber \\
&- \left(\beta^r +\mu - \frac{1}{2\gamma^{r-1}} - \frac{\widetilde{L}_{x}^{r-1}}{2}\right) \| x^{r+1} - x^{r} \|^{2}  \nonumber \\
& \hspace{-3mm}\leq  -\left(\beta^r - \frac{1+d_1}{2\gamma^{r-1}}- d_2\right) \| x^{r+1} - x^{r} \|^{2} + \frac{\gamma^{r-1} \widetilde{L}_{x}^{2}}{2}\sum\limits_{i=1}^{K} \| {\hat{y}^{r}_i} - y^{r} \|^{2},
\end{align}}%
where we have defined
\begin{align}
\widetilde{L}_{x} & := \max_{i}L_{x_i}, \quad  d_1 := L_{y}\widetilde{L}_{x}, \quad d_2 := \widetilde{L}_{x} - \mu\\
\widetilde{L}_{x}^{r-1}& :=2\widetilde{L}_x +\frac{L_y \widetilde{L}_x}{\gamma^{r-1}}. 
\end{align}
According to the above definition, we further 
$$\widetilde{L}_{x}^{r-1}\ge 2L_{x_i} +\frac{L_y L_{x_i}}{\gamma^{r-1}}=\widetilde{L}_{x_i}^{r-1}, \forall i.$$

In order to proceed we need to bound the term $\| {\hat{y}^{r}_i} - y^{r} \|$, where {$\hat{y}^{r}_i$} is defined in \eqref{def_yhat_i}. 
{First, following a reasoning similar to \cite[Lemma B.1]{nouiehed2019solving} and noticing that 
	$$\widetilde{y}^{r} - {\hat{y}^{r}_{i}} =  \arg \max\limits_{y \in \mathcal{Y}} r_{\gamma^{r-1}}(x^{r},y) - \arg \max\limits_{y \in \mathcal{Y}}\; r_{\gamma^{r-1}}(x_i^r,w_{-i}^{r+1},y),$$ 
	it is easy to obtain the following bound:}
\begin{align} \label{eq:eq6}
\|\widetilde{y}^{r} - {\hat{y}^{r}_{i}} \| \leq \left(1+\frac{ L_{y}}{\gamma^{r-1}}\right) \|x^{r+1}- x^r\|.
\end{align}
Then, using the triangle inequality, {\eqref{eq:eq6}  and the property $(a+b)^2\leq 2a^2 + 2b^2$} we obtain the following:
\begin{align}\label{eq:bound_y}
&\sum\limits_{i=1}^{K} \|\hat{y}^{r}_i -y^{r} \|^2 \nonumber\\
&=
\sum\limits_{i=1}^{K} \|\hat{y}^{r}_i -\widetilde{y}^{r} +\widetilde{y}^{r}-y^{r}\|^2  \nonumber\\
&\leq \sum\limits_{i=1}^{K} \left\lbrace 2\|\hat{y}^{r}_i -\widetilde{y}^{r} \|^2 + 2\|\widetilde{y}^{r}-y^{r}\|^2 \right\rbrace \nonumber \\
& \leq \sum\limits_{i=1}^{K} \left\lbrace 2\left(1+\frac{ L_{y}}{\gamma^{r-1}}\right)^2 \|x^{r+1}- x^r\|^2 + 2(\delta^{r})^2 \right\rbrace \nonumber \\
& \leq 2K\frac{ d_3}{(\gamma^{r-1})^2} \|x^{r+1}- x^r\|^2 + 2K(\delta^{r})^2
\end{align}
for some suitable constant $d_3 := d_3(\gamma_{0})>0$. {Note that the second to the last inequality is due to the definition of $\tilde{y}^r$ in \eqref{eq:y}.}

Recall that from Lemma \ref{lemma:phi_prop}, we have the following property:
$$\phi_{\gamma^{r}}(x) \leq \phi_{\gamma^{r-1}}(x)  + \frac{1}{2}(\gamma^{r-1}-\gamma^{r}) d^2.$$ 
Combine the above with 
{\eqref{eq:eq4} and \eqref{eq:bound_y}},  we obtain over overall estimate of the descent: %
\begin{align}\label{eq:phi_diff}
&\phi_{\gamma^{r}}(x^{r+1}) - \phi_{\gamma^{r-1}}(x^{r})\\
& {= \phi_{\gamma^{r}}(x^{r+1}) - \phi_{\gamma^{r-1}}(x^{r+1}) + \phi_{\gamma^{r-1}}(x^{r+1}) - \phi_{\gamma^{r-1}}(x^{r})} \nonumber \nonumber\\
& {\leq \frac{\gamma^{r-1}-\gamma^{r}}{2}d^2 -\left(\beta^r - \frac{1+d_1}{2\gamma^{r-1}}- d_2\right) \| x^{r+1} - x^{r} \|^{2}} \nonumber\\
&{\quad + \frac{\gamma^{r-1} \widetilde{L}_{x}^{2}}{2}\sum\limits_{i=1}^{K} \| {\hat{y}^{r}_i} - y^{r} \|^{2}} \nonumber\\
&\leq  K\gamma^{r-1} \widetilde{L}_{x}^{2} \left( \frac{{d_{3}}}{(\gamma^{r-1})^2} \|x^{r+1}- x^r\|^2 + (\delta^{r})^2   \right) \nonumber\\
&-\left(\beta^r - \frac{1+d_1}{2\gamma^{r-1}}- d_2\right)  \| x^{r+1} - x^{r} \|^{2} + \frac{\gamma^{r-1}-\gamma^{r}}{2}d^2 \nonumber\\
&\leq -\left( \beta^r - \frac{d_4}{2\gamma^{r-1}}- d_2 \right)\|x^{r+1} - x^{r} \|^2  + d_5 \gamma^{r-1}(\delta^{r})^2 \nonumber\\
& \quad +\frac{\gamma^{r-1}-\gamma^{r}}{2}d^2,
\end{align}%}%
where  $d_4 := 1+d_1+2K\widetilde{L}_{x}^{2}{d_{3}}>0$ and $d_5 :=	 K\widetilde{L}_{x}^{2}$.

{In order to ensure that the coefficient in front of the term $\|x^{r+1} - x^{r} \|^2$ is positive, let us pick
$$\beta^{r} = \frac{d_4}{\gamma^{r-1}}+ d_2.$$}
{Then, \eqref{eq:phi_diff} can be written as}
\begin{align}\label{eq:eq5}
&\frac{d_4}{2\gamma^{r-1}} \|x^{r+1} - x^{r} \|^2 \\
&\leq \phi_{\gamma^{r-1}}(x^{r})-\phi_{\gamma^{r}}(x^{r+1})+ d_5\gamma^{r-1}(\delta^{r})^2 +\frac{\gamma^{r-1}-\gamma^{r}}{2}d^2.\nonumber
\end{align}
Following the same steps as in \cite[Theorem 1]{lu2019hybrid_ieee} we obtain a bound on the $i$th {block} of the optimality condition {\eqref{eq.optgap}}, that is
$$
\| ({\nabla} {\mathcal{G}^{\beta^{r}}}(x^r,y^r))_{i}\|^{2}
\le \left(\beta^{r} + L \right)^{2}\|x^{r+1}-x^r\|^{2},
$$
where $L = L_{u_i} + L_{x_i}$.
Then, for the choice of $\beta^{r}$ adopted above and {by summing over the first $[K]$ blocks} we get
%\vspace{-2mm}
\begin{align} \label{eq:eq8}
\| ({\nabla} {\mathcal{G}^{\beta^{r}}}(x^r,y^r))_x\|^{2}
&\leq K \left(\frac{d_4}{\gamma^{r-1}}+ d_2 + L \right)^2\|x^{r+1}-x^r\|^2 \nonumber \\
&\leq \frac{d_{5}}{(\gamma^{r-1})^{2}}\|x^{r+1}-x^r\|^2
\end{align}%
for some constant $d_5$, {independent of the iteration number $r$}.
Moreover, we multiply both sides of \eqref{eq:eq8} with $\frac{d_4 \gamma^{r-1}}{2d_5}$  and combine it with \eqref{eq:eq5}, which yields
\begin{align*}
&\frac{d_4 \gamma^{r-1}}{2d_5} \| ({\nabla} \mathcal{G}^{\beta^{r}}(x^r,y^r))_x\|^{2} \leq \phi_{\gamma^{r-1}}(x^{r})-\phi_{\gamma^{r}}(x^{r+1})  \nonumber\\
&+\frac{\gamma^{r-1}-\gamma^{r}}{2}d^2 + d_5\gamma^{r-1}(\delta^{r})^2.
\end{align*}%
As a result, for a given $\epsilon$ if we sum over {$r = 1, \ldots, T$}, where $ T = T(\epsilon) \:= \min\{r \geq 1 \mid \|{{\nabla} {\mathcal{G}^{\beta^{r}}}}({ x^{r+1},y^{r+1}})\|\le\epsilon\}$, we get
%\vspace{-3mm}
\begin{align*}
&\sum\limits_{r=1}^{T} \frac{d_4 \gamma^{r-1}}{2d_5} \| ({\nabla} {\mathcal{G}^{\beta^{r}}}(x^r,y^r))_x\|^{2} \leq \phi_{\gamma^{0}}(x^{1}) - \phi_{\gamma^{T}}(x^{T+1})  \\
&+\frac{\gamma^{0}-\gamma^{T}}{2}d^2 + d_5\sum\limits_{r=1}^{T} \gamma^{r-1}(\delta^{r})^2
\end{align*}
%\vspace{-5mm}
which implies that
\begin{align}\label{eq:bound_prop1}
 \epsilon^{2} \leq \frac{d_6 +d_7\sum\limits_{r=1}^{T} \gamma^{r-1}(\delta^{r})^2}{\sum\limits_{r=1}^{T} \gamma^{r-1}},
\end{align}%
%\vspace{-1mm}
where we have defined 
{\small
$$d_6: =\frac{2d_5}{d_4}\left(\phi_{\gamma^{0}}(x^{1}) - \min\limits_{x \in \mathcal{X}}\{\phi_{\gamma^{0}}(x)\}+\frac{\gamma^{0}-\gamma^{T}}{2}d^2 \right), \; d_7: = \frac{2d_5^2}{d_4}.$$}
Notice that for $\gamma^{r}= \frac{1}{r^{1/3}}$  we have
\begin{align*}
\sum\limits_{r=1}^{T}  \gamma^{r-1} \geq \sum\limits_{r=1}^{T}  \gamma^{r} = \sum\limits_{r=1}^{T} \frac{1}{r^{1/3}} \geq T^{2/3}
\end{align*}
and for $\delta^{r} = \frac{1}{r^{1/3}}$ we see that
\begin{align*}
\sum\limits_{r=1}^{T} \gamma^{r-1}(\delta^{r})^2 \leq \sum\limits_{r=1}^{T} \frac{C^{\prime}}{r^{1/3}}\frac{1}{r^{2/3}} \leq \sum\limits_{r=1}^{T} \frac{C^{\prime}}{r} \leq C \ln(T),
\end{align*}%
where $C^{\prime}, C$ are some constants.  {Finally, for the above choices for $\gamma^{r}$ and  $\delta^{r}$ \eqref{eq:bound_prop1} becomes
\begin{align*}
\epsilon \leq \frac{(C \ln(T(\epsilon)))^{1/2}}{\left(T(\epsilon)\right)^{1/3}}.
\end{align*}
Then it follows that $T(\epsilon) = \widetilde{\mathcal{O}}\left(\frac{1}{\epsilon^3} \right)$.} \hfill $\blacksquare$
}
\end{proof}

\begin{proposition}[Analysis of the maximization problem] \label{prop.main2}
{\it Suppose that {Assumptions A -- C} hold. Then after the following iterations of the proximal gradient ascent method (with constant stepsize) we can obtain a solution $y^{r}$ of the maximization problem such {that \eqref{eq:y} holds true.
$${ J^{r}} \geq \log\left( \frac{d^{2}}{(\delta^{r})^{2}} \right)/\log\left(\frac{L_y+\gamma^{r-1}}{L_y}\right).$$
} 

}
\end{proposition}
\vspace{4mm}
\begin{proof}
First, let {$p(z)$} be a function  that {has} $\hat{L}$-Lipschitz gradient and is $\hat{\mu}$-strongly convex. Also, let {$q(z)$} be some non-smooth convex function. According to a standard result in optimization literature \cite{Schmidt14convergencerate} the application of the proximal gradient descent method on {$p(z) + q(z)$} with stepsize $\alpha=1/{\hat{L}}$ for $t$ steps yields $\|z^{t}-z^{\ast}\| \leq \left( 1-\frac{\hat{\mu}}{\hat{L}}\right)^{t} \| z^{0} - z^{\ast}\|$, where $z^{\ast}$ is the global minimum of the problem.

Let us apply the proximal gradient ascent method to the following problem 
$$\max\limits_{y \in \mathcal{Y}}\; r_{\gamma^{r-1}}(x^{r},y).$$ 
{Note that from Lemma \ref{lemma:constants} the objective has strong concave constant $\gamma^r$, and has Lipschitz constant $L_y+ \gamma^r$.}
So to ensure that $\| \widetilde{y}^{r} - y^{r}\| \leq \delta^{r}$ we must impose the condition
$$
\left( \frac{L_y}{\gamma^{r-1}+L_y} \right)^{{J^{r}}} d \leq \delta^{r},
$$
where $d$ is a bound on the norm of the elements of $\mathcal{Y}$ and ${J^{r}}$ is the number of  steps.  {Solving for $J^{r}$ results to} 
$${J^{r}} \geq \log\left( \frac{d}{\delta^{r}} \right)/\log\left(\frac{L_y+\gamma^{r-1}}{L_y}\right).$$

\hfill $\blacksquare$
\end{proof}

Finally, we give the total complexity of HiBSA.
\vspace{1mm}
{
\begin{theorem}[Gradient complexity]\label{th.main2}
{\it Suppose that {Assumptions A --  C} hold and let $(x^r,y^r)$ be a sequence generated by HiBSA. Then, in order to obtain an $\epsilon$-stationary point \eqref{eq:opt}, we need $T(\epsilon)$ iterations in the outer loop and $J^{r} = \log\left( \frac{d^{2}}{(\delta^{r})^{2}} \right)/\log\left(\frac{L_y+\gamma^{r-1}}{L_y}\right)$  iterations in the inner loop, where $\delta^{r} = \frac{1}{r^{1/3}}$, $\gamma^{r}= \frac{1}{r^{1/3}}$, $r=1,\ldots, T(\epsilon)$. Moreover, the total number of gradient computations required for reaching an $\epsilon$-stationary point \eqref{eq:opt} is $\widetilde{\mathcal{O}}(\frac{1}{\epsilon^{4}} \log\left( \frac{1}{\epsilon} \right))$.
}
\end{theorem}
}

\begin{proof}
{At each iteration $r$ (of the outer loop) we compute one gradient for the minimization problem and perform $J^r$ steps of gradient ascent for the maximization problem. From Proposition \ref{prop.main3}, the total number of gradient computations that is required in the outer iteration is $\widetilde{\mathcal{O}}(\frac{1}{\epsilon^{3}})$. Moreover, the total number of gradient computations for the inner maximization problem  is bounded by} 
\begin{align*}
&\sum\limits_{r=1}^{{T(\epsilon})} \frac{\log\left( \frac{d}{\delta^{r}} \right)}{\log\left(\frac{L_y+\gamma^{r-1}}{L_y}\right)} {\mathop{=}\limits^{(a)}}
\sum\limits_{r=1}^{\widetilde{\mathcal{O}}(1/{\epsilon^{3}})} \frac{\log\left( d r^{1/3} \right)}{\log\left( 1 + \frac{1}{L_y (r-1)^{1/3}}\right)}  \\
& \mathop{\leq}\limits^{(b)}  \sum\limits_{r=1}^{\widetilde{\mathcal{O}}(1/{\epsilon^{3}})} \log\left( d r^{1/3} \right) \left(L_y (r-1)^{1/3} +1 \right) \\
&{\mathop{\leq}\limits^{(c)}}  d_9 \sum\limits_{r=1}^{\widetilde{\mathcal{O}}(1/{\epsilon^{3}})} \log\left( r \right) r^{1/3} \\
&\sim \int\limits_{1}^{1/\epsilon^{3}} \log(r)r^{1/3} dr \sim \widetilde{\mathcal{O}}\left(\frac{1}{\epsilon^{4}}\log\left(\frac{1}{\epsilon} \right) \right)
\end{align*}%}%
where {in (a) we substituted $\delta^{r} = \frac{1}{r^{1/3}}$, $\gamma^{r}= \frac{1}{r^{1/3}}, T(\epsilon) = \widetilde{\mathcal{O}}(1/{\epsilon^{3}})$; in (b)} we used the property $\log(x) \geq 1-\frac{1}{x}$ and  {in (c)} $d_9$ is a large enough constant.
{Consequently, the total number of gradient computations required for attaining an $\epsilon$-stationary point \eqref{eq:opt} is $\widetilde{\mathcal{O}}(\frac{1}{\epsilon^{4}} \log\left( \frac{1}{\epsilon} \right))$}
\hfill $\blacksquare$
\end{proof}

\bibliographystyle{IEEEbib}
\bibliography{refs}